\documentclass[a4paper]{amsart}
\usepackage{amssymb}
\usepackage{amsthm}
\usepackage{amsmath,amscd}
\usepackage[mathscr]{euscript}
\usepackage[all]{xy}
\usepackage[english, french]{babel}
\usepackage{color}
\usepackage[utf8]{inputenc}
\normalfont
\usepackage[T1]{fontenc}
\usepackage[textwidth=14cm,hcentering]{geometry}
\usepackage[colorlinks=true,linkcolor=red,citecolor=blue]{hyperref}
\usepackage{enumerate}
\newtheorem{theo}{Theorem}[section]
\newtheorem{lemm}[theo]{Lemma}
\newtheorem{prop}[theo]{Proposition}
\newtheorem{coro}[theo]{Corollary}

\theoremstyle{definition}
\newtheorem{defi}[theo]{Definition}

\newtheorem{example}[theo]{Example}

\newtheorem{rem}[theo]{Remark}

\newtheorem*{theo*}{Theorem}

\numberwithin{equation}{section}

\makeatletter
\newenvironment{abstracts}{%
  \ifx\maketitle\relax
    \ClassWarning{\@classname}{Abstract should precede
      \protect\maketitle\space in AMS document classes; reported}%
  \fi
  \global\setbox\abstractbox=\vtop \bgroup
    \normalfont\Small
    \list{}{\labelwidth\z@
      \leftmargin3pc \rightmargin\leftmargin
      \listparindent\normalparindent \itemindent\z@
      \parsep\z@ \@plus\p@
      
      \itemsep\medskipamount
    }%
}{%
  \endlist\egroup
  \ifx\@setabstract\relax \@setabstracta \fi
}

\newcommand{\abstractin}[1]{%
  \otherlanguage{#1}%
  \item[\hskip\labelsep\scshape\abstractname.]%
}
\makeatother

\newcommand{\op}{^{\mathrm{op}}}
\newcommand{\on}{\operatorname}

\newcommand{\Hom}{\mathsf{Hom}}

\newcommand{\Aut}{\mathsf{Aut}}
\newcommand{\AUT}{\underline{\Aut}}
\newcommand{\Fun}{\on{Fun}}

\newcommand{\ho}{\mathrm{ho}}

\newcommand{\oper}{\mathcal}

\newcommand{\C}{\mathbb{C}}
\newcommand{\Q}{\mathbb{Q}}
\newcommand{\Z}{\mathbb{Z}}
\newcommand{\R}{\mathbb{R}}

\newcommand{\cat}{\mathrm}
\newcommand{\icat}{\mathbf}

\newcommand{\inn}{\icat{N}} 
\newcommand{\ch}{\cat{Ch}_*}
\newcommand{\chb}{\cat{Ch}_*^b}
\newcommand{\ich}{\icat{Ch}_*}

\newcommand{\ceil}[1]{\lceil#1\rceil}
\newcommand{\D}{\mathscr{D}} %holomorphic currents
\newcommand{\Ker}{\mathrm{Ker}}

\newcommand{\Img}{\mathrm{Im}}

\newcommand{\Vect}{\mathrm{Vect}}

\newcommand{\lra}{\longrightarrow}

\newcommand{\MHS}{\mathrm{MHS}}
\newcommand{\MHC}{\mathrm{MHC}}
\newcommand{\iMHC}{\mathbf{MHC}}

\newcommand{\Var}{\mathrm{Var}}
\newcommand{\Sm}{\mathrm{Sm}}

\newcommand{\chpure}[1]{\cat{Ch}_*(gr\Aa)^{#1\text{-pure}}}

\newcommand{\Aa}{\mathcal{A}}

\newcommand{\Cc}{\mathcal{C}}
\newcommand{\Dd}{\mathcal{D}}

\newcommand{\Ff}{\mathcal{F}}
\newcommand{\Gg}{\mathcal{G}}
\newcommand{\Hh}{\mathcal{H}}

\newcommand{\Kk}{\mathcal{K}}

\newcommand{\Mm}{\mathcal{M}}

\newcommand{\Tt}{\mathcal{T}}

\newcommand{\kk}{\Bbbk}

\newcommand{\TW}{\text{\tiny{$\mathrm{TW}$}}}

\title{Mixed Hodge structures and formality of symmetric monoidal functors}

\author{Joana Cirici}
\author{Geoffroy Horel}

\address{Deaprtament de Matem\`{a}tiques i Inform\`{a}tica\\ Universitat de Barcelona
Gran Via 585\\
08008 Barcelona}
\email{jcirici@ub.edu}

\address{LAGA, Institut Galilée, Universit\'{e} Paris 13\\ 99 avenue Jean-Baptiste Cl\'{e}ment\\ 93430 Villetaneuse\\ France}
\email{horel@math.univ-paris13.fr}

\thanks{Cirici would like to acknowledge financial support from the DFG (SPP-1786) and AGAUR (Beatriu de
Pinós Program) and partial support from the AEI/FEDER, UE (MTM2016-76453-C2-2-P).
Horel acknowledges support from the project ANR-16-CE40-0003 ChroK}


\begin{document}

%\begin{abstract}

\begin{abstracts}
\abstractin{english}
We use mixed Hodge theory to show that the functor of
singular chains with rational coefficients is formal as a lax symmetric monoidal functor, when restricted 
to complex varieties whose weight filtration in cohomology satisfies a certain purity property.
This has direct applications to the formality of operads or, more generally, of
algebraic structures encoded by a colored operad.
We also prove a dual statement, with applications to
formality in the context of rational homotopy theory. In the general case of complex varieties with non-pure weight filtration,
we relate the singular chains functor to a functor defined via the first term of the weight spectral sequence.

\abstractin{french}
Nous utilisons la théorie de Hodge mixte pour montrer que le foncteur des chaînes singulières à coefficients rationnels est formel, comme foncteur symétrique monoïdal lax, lorsqu’on le restreint aux variétés complexes dont la filtration par le poids en cohomologie satisfait une
certaine propriété de pureté. Ce résultat a des applications directes à la formalité d’opérades ou plus généralement à des structures algébriques encodées par une opérade colorée. Nous prouvons aussi le résultat dual, avec des applications à la formalité dans le contexte de la théorie de l’homotopie rationnelle. Dans le cas général d’une variété dont la filtration par le poids n'est pas pure, nous relions le foncteur des chaînes singulières à un foncteur défini par la première page de
la suite spectrale des poids.

\end{abstracts}

%\end{abstract}
\selectlanguage{english}

\maketitle

\section{Introduction}

There is a long tradition of using Hodge theory as a tool for proving formality results. The first instance of this idea can be found in \cite{DGMS} where the authors prove that compact K\"{a}hler manifolds are formal (i.e. the commutative differential graded algebra of differential forms is quasi-isomorphic to its cohomology). In the introduction of that paper, the authors explain that their intuition came from the theory of étale cohomology and the fact that the degree $n$ étale cohomology group of a smooth projective variety over a finite field is pure of weight $n$. This purity is what heuristically prevents the existence of non-trivial Massey products. In the setting of complex algebraic geometry, Deligne introduced in \cite{DeHII,DeHIII} a filtration on the rational cohomology of every complex algebraic variety $X$, called the \textit{weight filtration}, with the property that each of the successive quotients of this filtration behaves as the cohomology of a smooth projective variety, in the sense that it has a Hodge-type decomposition. Deligne's mixed Hodge theory was subsequently promoted to the rational homotopy of complex algebraic varieties (see \cite{Mo}, \cite{Ha}, \cite{Na}). This can then be used to make the intuition of the introduction of \cite{DGMS} precise. In \cite{Dupont} and \cite{ChCi1}, it is shown that purity of the weight filtration in cohomology implies formality, in the sense of rational homotopy, of the underlying topological space. However, the treatment of the theory in these references lacks functoriality and is restricted to smooth varieties in the first paper and to projective varieties in the second. 

In another direction, in the paper \cite{santosmoduli}, the authors elaborate on the method of \cite{DGMS} and prove that operads (as well as cyclic operads, modular operads, etc.) internal to the category of compact K\"{a}hler manifolds are formal. Their strategy is to introduce the functor of de Rham currents which is a functor from compact Kähler manifolds to chain complexes that is lax symmetric monoidal and quasi-isomorphic to the singular chain functor as a lax symmetric monoidal functor. Then they show that this functor is formal as a lax symmetric monoidal functor. Recall that, if $\Cc$ is a symmetric monoidal category and $\Aa$ is an abelian symmetric monoidal category, a lax symmetric monoidal functor $F:\Cc\lra \ch(\Aa)$ is said to be formal if it is weakly equivalent to $H\circ F$ in the category of lax symmetric monoidal functors. It is then straightforward to see that such functors send operads in $\Cc$ to formal operads in $\ch(\Aa)$. The functoriality also immediately gives us that a map of operads in $\Cc$ is sent to a formal map of operads or that an operad with an action of a group $G$ is sent to a formal operad with a $G$-action. Of course, there is nothing specific about operads in these statements and they would be equally true for monoids, cyclic operads, modular operads, or more generally any algebraic structure that can be encoded by a colored operad.

The purpose of this paper is to push this idea of formality of symmetric monoidal functors from complex algebraic varieties in several directions in order to prove the most general possible theorem of the form ``purity implies formality''. Before explaining our results more precisely, we need to introduce a bit of terminology.

Let $X$ be a complex algebraic variety. Deligne's weight filtration on the rational $n$-th cohomology vector space of $X$ is bounded by 
\[0=W_{-1}H^n(X,\Q)\subseteq W_{0}H^n(X,\Q)\subseteq \cdots \subseteq W_{2n}H^n(X,\Q)=H^n(X,\Q).\]
If $X$ is smooth then $W_{n-1}H^n(X,\Q)=0$, while if $X$ is projective $W_{n}H^n(X,\Q)=H^n(X,\Q)$.
In particular, if $X$ is smooth and projective then
we have 
\[0=W_{n-1}H^n(X,\Q)\subseteq W_{n}H^n(X,\Q)=H^n(X,\Q).\]
In this case, the weight filtration on $H^n(X,\Q)$ is said to be \textit{pure of weight $n$}.
More generally, for $\alpha$ a rational number and $X$ a complex algebraic variety, we say that the weight filtration on $H^*(X,\Q)$ is \textit{$\alpha$-pure} if, for all $n\geq 0$, we have
\[Gr^W_pH^n(X,\Q):={{W_pH^n(X,\Q)}\over{W_{p-1}H^n(X,\Q)}}=0\text{ for all }p\neq \alpha n.\]
The bounds on the weight filtration tell us that this makes sense only when $0\leq \alpha\leq 2$.
Note as well that if we write $\alpha=a/b$ with $(a,b)=1$, $\alpha$-purity implies that the cohomology is concentrated in degrees that are divisible 
by $b$, and that $H^{bn}(X,\Q)$ is pure of weight $an$.

Aside from smooth projective varieties, some well-known examples of varieties 
with $1$-pure weight filtration are:
projective varieties whose underlying topological space is a $\Q$-homology manifold (\cite[Theorem 8.2.4]{DeHIII})
and the moduli spaces $\Mm_{Dol}$ and $\Mm_{dR}$ appearing in the non-abelian Hodge correspondence (\cite{Hausel}).
Some examples of varieties with $2$-pure weight filtration are:
complements of hyperplane arrangements (\cite{kimweights}), which include the
moduli spaces
$\Mm_{0,n}$ of smooth projective curves of genus 0 with $n$ marked points,
and complements of toric arrangements (\cite{Dupont}). As we shall see in Section $\ref{Section_formalschemes}$, complements of codimension $d$ subspace arrangements are examples of smooth varieties whose weight filtration in cohomology is $2d/(2d-1)$-pure. For instance, this includes configuration spaces of points in $\C^d$.

Our main result is Theorem $\ref{theo: main covariant}$. We show that, for a non-zero rational number $\alpha$, the 
singular chains functor 
\[S_*(-,\Q):\Var_\C\lra \ch(\Q)\]
is formal as a lax symmetric monoidal 
functor when restricted to complex varieties whose weight filtration in cohomology is $\alpha$-pure.
Here $\Var_\C$ denotes the category of complex algebraic varieties (i.e the category of schemes over $\C$ that are reduced, separated and of finite type).
This generalizes the main result of \cite{santosmoduli} on the formality of $S_*(X,\Q)$ for
any operad $X$ in smooth projective varieties, to the case of operads in possibly singular and/or non-compact varieties
with pure weight filtration in cohomology.

As direct applications of the above result, we prove formality of the operad of
singular chains of some operads in complex varieties, such as the noncommutative analog of the (framed) little 2-discs operad introduced in \cite{dotsenkononcommutative} and the monoid of self-maps of the complex projective line studied by Cazanave in \cite{cazanavealgebraic}
(see Theorems $\ref{theo: little}$ and $\ref{theo: monoidformal}$).
We also reinterpret in the language of mixed Hodge theory the proofs of the formality of the little disks operad and Getzler's gravity operad appearing in \cite{Petersen} and \cite{dupontgravity}. These last two results do not fit directly in our framework, since the little disks operad and the gravity operad do not quite come from operads in algebraic varieties. However,
the action of the Grothendieck-Teichmüller group provides a bridge to mixed Hodge theory.

In Theorem \ref{theo: main contravariant} we prove a dual statement of our main result, 
showing that Sullivan's functor of piecewise linear forms
\[\Aa_{PL}^*:\Var_\C\op\lra \ch(\Q)\]
is formal as a lax symmetric monoidal functor when restricted to varieties whose weight filtration in 
cohomology is $\alpha$-pure, where $\alpha$ is a non-zero rational number.

This gives functorial formality in the sense of rational homotopy for such varieties,
generalizing both ``purity implies formality'' statements appearing in \cite{Dupont} for smooth varieties
and in \cite{ChCi1} for singular projective varieties.
Our generalization is threefold: we allow $\alpha$-purity (instead of just $1$- and $2$-purity), we obtain functoriality and 
we study possibly singular and open varieties simultaneously.

Theorems $\ref{theo: main covariant}$ and $\ref{theo: main contravariant}$ deal with situations in which the weight filtration is pure.
In the general context with mixed weights, it was shown by Morgan \cite{Mo} for smooth varieties and in \cite{CG1} for possibly singular varieties, that the first term of the multiplicative weight spectral sequence 
carries all the rational homotopy information of the variety.
In Theorem $\ref{E1formality}$ we provide the analogous statement for the lax symmetric monoidal functor of singular chains.
A dual statement for Sullivan's functor of piecewise linear forms is proven in Theorem $\ref{E1formality_contravariant}$,
enhancing the results of \cite{Mo} and \cite{CG1} with functoriality.
\\

We now explain the structure of this paper. 
The first four sections are purely algebraic.
In Section $\ref{Section_formal_functors}$ we collect the main properties
of formal lax symmetric monoidal functors that we use. In particular, in Theorem \ref{theo: Hinich rigidification} we recall a recent theorem of rigidification due to Hinich that says that, over a field of characteristic zero, formality of functors can be checked at the level of $\infty$-functors.
We also introduce the notion of $\alpha$-purity for complexes of bigraded objects in a symmetric monoidal abelian category and
show that, when restricted to $\alpha$-pure complexes, the functor defined by forgetting the degree is formal.

The connection of this result with mixed Hodge structures is done in Section $\ref{Section_MHS}$, where we prove a symmetric monoidal version of Deligne's weak splitting of mixed Hodge structures over $\C$. 
Such splitting is a key tool towards formality.
In Section $\ref{Section_descent}$ we study lax symmetric monoidal functors to vector spaces over 
a field of characteristic zero equipped with a compatible filtration. We show, in Theorem $\ref{descens_rsplittings}$, that the existence of a lax symmetric monoidal splitting for such functors can be verified after extending the scalars to a larger field. As a consequence, we obtain splittings for the weight filtration over $\Q$. This result enables us to
bypass the theory of descent of formality for operads of \cite{santosmoduli}, which assumes the existence of minimal models.
Putting the above results together we are able to show that the forgetful functor 
\[\ch(\MHS_\Q)\lra \ch(\Q)\]
induced by sending a rational mixed Hodge structure to its underlying vector space is formal when restricted to those complexes whose mixed Hodge structure in homology is $\alpha$-pure.

In order to obtain a symmetric monoidal functor from the category of complex varieties to an algebraic category encoding mixed Hodge structures,
we have to consider more flexible objects than complexes of mixed Hodge structures. This is the content of Section 
 $\ref{Section_MHC}$, where we study the category $\MHC_\kk$ of mixed Hodge complexes. In Theorem $\ref{theo: equivinfty}$ we construct an equivalence of symmetric monoidal $\infty$-categories between mixed Hodge complexes and complexes of mixed Hodge structures. This result is a lift of Beilinson's equivalence of triangulated categories 
$D^b(\MHS_\kk)\lra \ho(\MHC_\kk)$ (see also \cite{Drew}, \cite{BNT}). 

The geometric character of this paper comes in Section $\ref{Section_logaritmic}$, where we construct
a symmetric monoidal functor from complex varieties to mixed Hodge complexes. This is done in two steps.
First, for smooth varieties, we dualize Navarro's construction \cite{Na} of functorial mixed Hodge complexes to obtain
a symmetric monoidal $\infty$-functor 
\[\D_*:\inn(\cat{Sm}_{\mathbb{C}})\lra \iMHC_\Q\]
such that its composite with
the forgetful functor $\iMHC_\Q\lra \ich(\Q)$
is naturally weakly equivalent to $S_*(-,\Q)$ as a symmetric monoidal $\infty$-functor (see Theorem $\ref{theo: equivalence cochains-sheaf}$). 
Note that in order to obtain monoidality, we move to the world of $\infty$-categories, denoted in boldface letters.
In the second step, we extend this functor from smooth, to singular varieties, by standard cohomological descent arguments.

The main results of this paper are stated and proven in Section $\ref{Section_main}$, where we also explain 
several applications to operad formality. Lastly, Section $\ref{Section_formalschemes}$
contains applications to the rational homotopy theory of complex varieties.

\subsection*{Acknowledgments}
This project was started during a visit of the first author at the Hausdorff Institute for Mathematics as part of the Junior Trimester Program in Topology. We would like to thank the HIM for its support. We would also like to thank Alexander Berglund, Brad Drew, Clément Dupont, Vicen\c{c} Navarro, Thomas Nikolaus and Bruno Vallette for helpful conversations. Finally we thank the anonymous referees for many useful comments.

\subsection*{Notations}

As a rule, we use boldface letters to denote $\infty$-categories and normal letters to denote $1$-categories. For $\Cc$ a $1$-category, we denote by $\inn(\Cc)$ its nerve seen as an $\infty$-category.

For $\Aa$ an additive category, we will denote by $\ch^{?}(\Aa)$ the category of (homologically graded) chain complexes in $\Aa$, where 
$?$ denotes the boundedness condition: nothing for unbounded, $b$ for bounded below and above and $\geq 0$ (resp. $\leq 0$)
for non-negatively (resp. non-positively) graded complexes. We denote by $\ich^?(\Aa)$ the $\infty$-category obtained from $\ch^?(\Aa)$ by inverting the quasi-isomorphisms.

\section{Formal symmetric monoidal functors}\label{Section_formal_functors}
The main result of this section is a ``purity implies formality'' statement in the setting of symmetric monoidal functors.

Let $(\Aa,\otimes,\mathbf{1})$ be an abelian symmetric monoidal category with infinite direct sums. The homology functor $H:\ch(\Aa)\lra \prod_{n \in\Z}\Aa$ is a lax symmetric monoidal functor, via the usual K\"{u}nneth morphism. In the cases that will interest us, all the objects of $\Aa$ will be flat and the homology functor is in fact strong symmetric monoidal. We will also make the small abuse of identifying the category $\prod_{n \in\Z}\Aa$ with the full subcategory of $\ch(\Aa)$ spanned by the chain complexes with zero differential. 

We recall the following definition from \cite{santosmoduli}.

\begin{defi}\label{defFormalFunctor}
Let $\Cc$ be a symmetric monoidal category and
$F:\Cc\lra \ch(\Aa)$ a lax symmetric monoidal functor.
Then $F$ is said to be a \textit{formal lax symmetric monoidal} functor if $F$ and
$H\circ F$ are \textit{weakly equivalent} in the category of lax symmetric monoidal functors:
there is a string of natural transformations of lax symmetric monoidal functors 
\[F\xleftarrow{\Phi_1}F_1\longrightarrow \cdots \longleftarrow F_n\xrightarrow{\Phi_n}H\circ F\]
such that for every object $X$ of $\Cc$, the morphisms $\Phi_i(X)$ are quasi-isomorphisms.
\end{defi}

\begin{defi}
Let $\Cc$ be a symmetric monoidal category and $F:\inn(\Cc)\to\ich(\Aa)$ a lax symmetric monoidal functor (in the $\infty$-categorical sense). We say that $F$ is a \textit{formal lax symmetric monoidal $\infty$-functor} if $F$ and $H\circ F$ 
are equivalent in the $\infty$-category of lax symmetric monoidal functors from $\inn(\Cc)$ to $\ich(\Aa)$.
\end{defi}

Clearly a formal lax symmetric monoidal functor $\Cc\to \ch(\Aa)$ induces a formal lax symmetric monoidal $\infty$-functor $\inn(\Cc)\to \ich(\Aa)$. The following theorem and its corollary give a partial converse.

\begin{theo}[Hinich]\label{theo: Hinich rigidification}
Let $\kk$ be a field of characteristic $0$. Let $\Cc$ be a small symmetric monoidal category. Let $F$ and $G$ be two lax symmetric monoidal functors $\Cc\to\ch(\kk)$. If $F$ and $G$ are equivalent as lax symmetric monoidal $\infty$-functors $\inn(\Cc)\lra \ich(\kk)$, then $F$ and $G$ are weakly equivalent as lax symmetric monoidal functors. 
\end{theo}

\begin{proof}
This theorem is true more generally if $\Cc$ is a colored operad. Indeed recall that any symmetric monoidal category has an underlying colored operad whose category of algebras is equivalent to the category of lax symmetric monoidal functors out of the original category. 

Now since we are working in characteristic zero, the operad underlying $\Cc$ is homotopically sound (following the terminology of \cite{hinichrectification}). Therefore, \cite[Theorem 4.1.1]{hinichrectification} gives us an equivalence of $\infty$-categories
\[\inn(\cat{Alg}_\Cc(\ch(\kk))\xrightarrow{\sim} \icat{Alg}_\Cc(\ich(\kk))\]
where we denote by $\cat{Alg}_\Cc$ (resp. $\icat{Alg}_\Cc$) the category of lax symmetric monoidal functors (resp. the $\infty$-category of lax symmetric monoidal functors) out of $\Cc$. Now, the two functors $F$ and $G$ are two objects in the source of the above map that become weakly equivalent in the target. Hence, they are already equivalent in the source, which is precisely saying that they are connected by a zig-zag of weak equivalences of lax symmetric monoidal functors.
\end{proof}

\begin{coro}\label{coro: Hinich rigidification}
Let $\kk$ be a field of characteristic $0$. Let $\Cc$ be a small symmetric monoidal category. Let $F:\Cc\to\ch(\kk)$ be a lax symmetric monoidal functor. If $F$ is formal as lax symmetric monoidal $\infty$-functor $\inn(\Cc)\lra \ich(\kk)$, then $F$ is formal as a lax symmetric monoidal functor. 
\end{coro}

\begin{proof}
It suffices to apply Theorem \ref{theo: Hinich rigidification} to $F$ and $H\circ F$.
\end{proof}

The following proposition whose proof is trivial is the reason we are interested in formal lax monoidal functors.

\begin{prop}[\cite{santosmoduli}, Proposition 2.5.5]
If $F:\Cc\lra \ch(\Aa)$ is a formal lax symmetric monoidal functor then $F$ sends operads in $\Cc$ to formal operads in $\ch(\Aa)$.
\end{prop}

In rational homotopy, there is a criterion of formality in terms of weight decompositions
which proves to be useful in certain situations (see for example \cite{BMSS} and \cite{Body-Douglas}).
We next provide an analogous criterion in the setting of symmetric monoidal functors.

Denote by $gr\Aa$ the category of graded objects of $\Aa$. 
It inherits a symmetric monoidal structure from that of $\Aa$, with the tensor product defined by 
\[(A\otimes B)^n:=\bigoplus_p A^p\otimes B^{p-n}.\]
The unit in $gr\Aa$ is given by $\mathbf{1}$ concentrated in degree zero. The functor $U:gr\Aa\lra \Aa$ obtained by forgetting the degree is strong symmetric monoidal. The category of graded complexes $\ch(gr\Aa)$ inherits a symmetric monoidal structure
via a graded K\"{u}nneth morphism.

\begin{defi}\label{defpure}
Given a rational number $\alpha$, denote by $\chpure{\alpha}$ the full subcategory of $\ch(gr\Aa)$ given by those graded complexes $A=\bigoplus A_n^p$ with \textit{$\alpha$-pure homology}:
\[H_n(A)^p=0\text{ for all }p\neq \alpha n.\]
\end{defi}
Note that if $\alpha=a/b$, with $a$ and $b$ coprime, then the above condition implies that $H_*(A)$ 
is concentrated in degrees that are divisible by $b$, and in degree $kb$, it is pure of weight $ka$:
\[H_{kb}(A)^p=0\text{ for all }p\neq ka.\]

\begin{prop}\label{decomposition_formal}
Let $\Aa$ be an abelian category and $\alpha$ a non-zero rational number.
The functor $U:\chpure{\alpha}\lra \ch(\Aa)$ defined by forgetting the degree is formal as a lax symmetric monoidal functor.
\end{prop}

\begin{proof}
We will define a functor $\tau:\ch(gr\Aa)\lra \ch(gr\Aa)$ together with natural transformations
\[\Phi:U\circ \tau\Rightarrow U\text{ and }\Psi:U\circ\tau\Rightarrow H\circ U\]
giving rise to weak equivalences when restricted to chain complexes with $\alpha$-pure homology.

Consider the truncation functor $\tau:\ch(gr\Aa)\lra \ch(gr\Aa)$ defined by sending 
a graded chain complex $A=\bigoplus A_n^p$ to the graded complex given by:
\[(\tau A)_n^p:=\left\{ 
\begin{array}{ll}
A_n^p& n>\ceil{p/\alpha}\\
\Ker(d:A_n^p\to A_{n-1}^p)& n=\ceil{p/\alpha}\\
0& n<\ceil{p/\alpha}
\end{array}\right.,
\]
where $\ceil{p/\alpha}$ denotes the smallest integer greater than or equal to $p/\alpha$.
Note that for each $p$, $\tau(A)^p_*$ is the chain complex given by the canonical truncation of $A_*^p$ at $\ceil{p/\alpha}$,
which satisfies 
\[H_n(\tau(A)^p_*)\cong H_n(A^p_*)\text{ for all }n\geq \ceil{p/\alpha}.\]
To prove that $\tau$ is a lax symmetric monoidal functor it suffices to see that 
\[\tau(A)^p_n\otimes \tau(B)^q_m\subseteq \tau(A\otimes B)^{p+q}_{n+m}\]
for all $A,B\in\ch(gr\Aa)$.
By symmetry in $A$ and $B$, it suffices to consider the following three cases : 
\begin{enumerate}
 \item If $n>\ceil{p/\alpha}$ and $m\geq \ceil{q/\alpha}$ then
 $n+m>\ceil{p/\alpha}+\ceil{q/\alpha}\geq \ceil{(p+q)/\alpha}$.
 Therefore we have $\tau(A\otimes B)^{p+q}_{n+m}=(A\otimes B)^{p+q}_{n+m}$ and the above inclusion is trivially satisfied.
\item If $n=\ceil{p/\alpha}$ and $m=\ceil{q/\alpha}$ then $n+m=\ceil{p/\alpha}+\ceil{q/\alpha}\geq \ceil{(p+q)/\alpha}$.
 Now, if $n+m>\ceil{(p+q)/\alpha}$ then again we have 
 $\tau(A\otimes B)^{p+q}_{n+m}=(A\otimes B)^{p+q}_{n+m}$. If $n+m=\ceil{(p+q)/\alpha}$ then 
 the above inclusion reads
 \[\Ker(d:A_n^p\to A_{n-1}^p)\otimes \Ker(d:B_m^q\to B_{m-1}^q)\subseteq \Ker(d:(A\otimes B)_{n+m}^{p+q}\to (A\otimes B)_{n+m-1}^{p+q}).\]
This is verified by the Leibniz rule.
\item Lastly, if $n<\ceil{p/\alpha}$ then $\tau(A)^p_n=0$ and there is nothing to verify.
\end{enumerate}

The projection 
$\Ker(d:A_n^p\to A_{n-1}^p)\twoheadrightarrow H_n(A)^p$ defines a morphism $\tau A\to H(A)$  by
\[(\tau A)_n^p\mapsto 
\left\{ 
\begin{array}{ll}
0& n\neq \ceil{p/\alpha}\\
H_n(A)^p& n=\ceil{p/\alpha}\\
\end{array}
\right..
\]
This gives a symmetric monoidal natural transformation $\Psi:U\circ\tau\Rightarrow H\circ U=U\circ H$.
Likewise, the inclusion $\tau A\hookrightarrow A$  defines a symmetric monoidal natural transformation $\Phi:U\circ\tau\Rightarrow U$.

Let $A$ be a complex of $\chpure{\alpha}$. Then both morphisms
\[\Psi(A):\tau\circ U(A)\to H\circ U(A)\text{ and }\Phi(A):U\circ\tau(A)\to U(A)\]
are clearly quasi-isomorphisms.
\end{proof}

For graded chain complexes whose homology is pure up to a certain degree, we obtain a result of partial formality as follows.

\begin{defi}
Let $q\geq 0$ be an integer. A morphism of chain complexes $f:A\to B$ is called a \textit{$q$-quasi-isomorphism} if the
induced morphism in homology $H_i(f):H_i(A)\to H_i(B)$ is an isomorphism for all $i\leq q$.
\end{defi}

\begin{rem}
There is a notion of $q$-quasi-isomorphism in rational homotopy which asks in addition 
that the map induced in degree $(q+1)$-cohomology is a monomorphism. Dually, for chain complexes one could ask to have an epimorphism in degree $(q+1)$-homology.
Note that we don't consider this extra condition here, since we work
with possibly negatively and positively graded complexes and such a condition would break the symmetry.
In addition, in our subsequent work on formality with torsion coefficients \cite{CiHo2},
the notion of partial formality as
defined below plays a fundamental role.
\end{rem}

\begin{defi}
Let $q\geq 0$ be an integer. A functor
$F:\Cc\lra \ch(\Aa)$ is a \textit{$q$-formal lax symmetric monoidal} functor if there are natural transformations
 $\Phi_i$ as in Definition $\ref{defFormalFunctor}$ such that $\Phi_i(X)$ are $q$-quasi-isomorphisms for all $X\in\Cc$ and all $1\leq i\leq n$.
\end{defi}

\begin{prop}\label{decomposition_qformal}
Let $\Aa$ be an abelian category. Given a non-zero rational number $\alpha$ and an integer $q\geq 0$, 
denote by $\chpure{\alpha}_q$ the full subcategory of $\ch(gr\Aa)$ 
given by those graded complexes $A=\bigoplus A_n^p$ whose homology in degrees $\leq q$ is $\alpha$-pure: 
 for all $n\leq q$,
\[H_n(A)^p=0\text{ for all }p\neq \alpha n.\]
Then the functor
$U:\chpure{\alpha}_q\lra \ch(\Aa)$ defined by forgetting the degree
is $q$-formal.
\end{prop}

\begin{proof}
The proof is parallel to that of Proposition \ref{decomposition_formal} by noting that, if 
$H_n(A)$ is $\alpha$-pure for $n\leq q+1$, then the 
morphisms
\[\Psi(A):\tau\circ U(A)\to H\circ U(A)\text{ and }\Phi(A):U\circ\tau(A)\to U(A)\] 
are $q$-quasi-isomorphisms.

\end{proof}

\section{Mixed Hodge structures}\label{Section_MHS}
We next collect some main definitions and properties on mixed Hodge structures and prove a symmetric monoidal version of Deligne's splitting for the weight filtration.

Denote by $\Ff\Aa$ the category of filtered objects of an abelian symmetric monoidal category $(\Aa,\otimes,\mathbf{1})$. 
All filtrations will be assumed to be of finite length and exhaustive.
With the tensor product
\[W_p(A\otimes B):=\sum_{i+j=p} \Img(W_iA\otimes W_jB\lra A\otimes B),
\]
and the unit given by $\mathbf{1}$ concentrated in weight zero, $\Ff\Aa$ is a symmetric monoidal category.
The functor $U^{fil}:gr\Aa\lra \Ff\Aa$ defined by
$A=\bigoplus A^p\mapsto W_mA:=\bigoplus_{q\leq m}A^q$ is strong symmetric monoidal.
The category of filtered complexes
$\ch(\Ff\Aa)$ inherits a symmetric monoidal structure via a filtered K\"{u}nneth morphism
and we have a strong symmetric monoidal functor
\[U^{fil}:\ch(gr\Aa)\lra \ch(\Ff\Aa).\]

Let $\kk\subset \R$ be a subfield of the real numbers.

\begin{defi}\label{defMHS}
A \textit{mixed Hodge structure} on a finite dimensional $\kk$-vector space $V$ is given by an increasing filtration $W$ of $V$, called the \textit{weight filtration},
together with a decreasing filtration $F$ on $V_\C:=V\otimes\C$, called the 
\textit{Hodge filtration}, such that  for all $m\geq 0$, each $\kk$-vector space $Gr_m^WV:=W_mV/W_{m-1}V$ 
carries a pure Hodge structure of weight $m$ given by the filtration induced by $F$ on $Gr_m^WV\otimes\C$, that is, there is a direct sum decomposition
\[Gr^m_WV\otimes\C=\bigoplus_{p+q=m}V^{p,q}\text{ where }V^{p,q}=F^p(Gr_m^WV\otimes\C)\cap \overline{F}^q(Gr_m^WV\otimes\C)=\overline{V}^{q,p}.\]
\end{defi}

Morphisms of mixed Hodge structures are given by morphisms $f:V\to V'$ of $\kk$-vector spaces compatible with filtrations: $f(W_mV)\subset W_mV'$ and $f(F^pV_\C)\subset F^pV_\C'$. 
Denote by $\MHS_\kk$ the category of mixed Hodge structures over $\kk$.  
It is an abelian category by \cite[Theorem 2.3.5]{DeHII}. 

\begin{rem}\label{Homandtensor}
Given mixed Hodge structures $V$ and $V'$, then $V\otimes V'$ carries a mixed Hodge structure with the filtered tensor product.
This makes
$\MHS_\kk$ into a symmetric monoidal category.
Also,  $\Hom(V,V')$ carries a mixed Hodge structure with the weight filtration given by
\[W_p\Hom(V,V'):=\{f:V\to V';f(W_qV)\subset W_{q+p}V',\,\forall q\}\]
and the Hodge filtration defined in the same way. In particular,
the dual of a mixed Hodge structure is again a mixed Hodge structure. 
\end{rem}

Let $\kk\subset\mathbb{K}$ be a field extension.
The functors 
\[\Pi_\mathbb{K}:\MHS_\kk\lra \Vect_\mathbb{K}\text{ and }\Pi_\mathbb{K}^{W}:\MHS_\kk\lra \Ff\Vect_\mathbb{K}\]
defined by sending a mixed Hodge structure $(V,W,F)$ to $V_\mathbb{K}:=V_\kk\otimes \mathbb{K}$ and $(V_\mathbb{K},W)$ respectively, are strong symmetric monoidal functors.

Deligne introduced a global decomposition of $V_\C:=V\otimes\C$
into subspaces $I^{p,q}$, for any mixed Hodge structure $(V,W,F)$ which generalizes the decomposition of 
pure Hodge structures of a given weight. In this case, one has a congruence
$I^{p,q}\equiv\overline{I}^{q,p}$ modulo $W_{p+q-2}$.
From this decomposition, Deligne deduced that morphisms of mixed Hodge structures are strictly 
compatible with filtrations and that the category of mixed Hodge structures is abelian (see \cite[Section 1]{DeHII}, see also \cite[Section 3.1]{PS}). We next study this decomposition in the context of symmetric monoidal functors.

\begin{lemm}[Deligne's splitting]\label{Deligne_splitting}
The functor $\Pi_\C^{W}$ admits a factorization
\[
\xymatrix@R=4pc@C=4pc{
\MHS_\kk\ar[r]^{G}\ar[dr]_{\Pi_\C^{W}}&gr\Vect_\C\ar[d]^{U^{fil}}\\
&\Ff\Vect_\C
}
\]
into strong symmetric monoidal functors. In particular, there is an isomorphism of functors
\[U^{fil}\circ gr\circ \Pi_\C^{W}\cong \Pi_\C^{W},\]
where $gr:\Ff\Vect_\C\lra gr\Vect_\C$ is the graded functor given by $gr(V_\C,W)^p=Gr_p^WV_\C$. 
\end{lemm}

\begin{proof}
Let $(V,W,F)$ be a mixed Hodge structure. By \cite[1.2.11]{DeHII} (see also \cite[Lemma 1.12]{GS}), there is a direct sum decomposition
$V_\C=\bigoplus I^{p,q}(V)$ where 
\[I^{p,q}(V)=(F^p W_{p+q}V_\C)\cap \left(\overline{F}^q W_{p+q}V_\C+\sum_{i>0}\overline{F}^{q-i} W_{p+q-1-i}V_\C\right).\]
This decomposition is functorial for morphisms of mixed Hodge structures
and satisfies 
\[W_mV_\C=\bigoplus_{p+q\leq m} I^{p,q}(V).\]
Define $G$ by letting $G(V,W,F)^n:=\bigoplus_{p+q=n} I^{p,q}(V)$ for any mixed Hodge structure.  Since $f(I^{p,q}(V))\subset I^{p,q}(V')$ for every morphism $f:(V,W,F)\to (V',W,F)$ of mixed Hodge structures, $G$ is functorial.
To see that $G$ is strong symmetric monoidal it suffices to use the definition of $I^{p,q}$ together with the 
tensor product mixed Hodge structure defined via the filtered tensor product, to obtain isomorphisms
\[\sum_{\substack{p+q=n\\p'+q'=n'}} I^{p,q}(V)\otimes I^{p',q'}(V')\cong 
\sum_{i+j=n+n'}I^{i,j}(V\otimes V')
\]
showing that the splittings $I^{*,*}$ are compatible with tensor products (see also \cite[Proposition 1.9]{Mo}).

The functor $U^{fil}:gr\Vect\lra \Ff\Vect$ is the strong symmetric monoidal functor given by
$$\bigoplus_n V^n\mapsto (V,W),\text{ with }W_mV:=\bigoplus_{n\leq m} V^n.$$
Therefore we have $U^{fil}\circ G=\Pi_\C^{W}$. In order to prove the isomorphism 
$U^{fil}\circ gr\circ \Pi_\C^{W}\cong \Pi_\C^{W}$ it suffices to note that 
there is an isomorphism of functors
$gr\circ U^{fil}\cong \mathrm{Id}.$
\end{proof}

\section{Descent of splittings of lax symmetric monoidal functors}\label{Section_descent}

In this section, we study lax symmetric monoidal functors to vector spaces over a field of characteristic zero $\kk$ equipped with a compatible filtration. More precisely, we are interested in lax symmetric  monoidal maps $\Cc\lra\Ff\Vect_\kk$. Our goal is to prove that the existence of a lax symmetric monoidal splitting for such a functor (i.e. of a lift of this map to $\Cc\lra gr\Vect_\kk$) can be checked after extending the scalars to a larger field. Our proof follows similar arguments to those appearing in \cite[Section 2.4]{CG1}, see also \cite{santosmoduli} and \cite{Su}.
A main advantage of our approach with respect to these references is that, in proving descent at the level of functors,
we avoid the use of minimal models (and thus restrictions to, for instance, operads with trivial arity 0).

It will be a bit more convenient to study a more general situation where $\Cc$ is allowed to be a colored operad instead of a symmetric monoidal category. Indeed recall that any symmetric monoidal category can be seen as an operad whose colors are the objects of $\Cc$ and where a multimorphism from $(c_1,\ldots,c_n)$ to $d$ is just a morphism in $\Cc$ from $c_1\otimes\ldots\otimes c_n$ to $d$. Then, given another symmetric monoidal category $\Dd$, there is an equivalence of categories between the category of lax symmetric monoidal functors from $\Cc$ to $\Dd$ and the category of $\Cc$-algebras in the symmetric monoidal category $\Dd$.

We fix $(V,W)$ a map of colored operads $\Cc\lra \Ff\Vect_\kk$ such that for each color $c$ of $\Cc$, the vector space $V(c)$ is finite dimensional. We denote by $\Aut_W(V)$ the set of its automorphisms in the category of $\Cc$-algebras in $\Ff\Vect_\kk$ and by $\Aut(Gr^W V)$ the set of automorphisms of $Gr^WV$ in the category of $\Cc$-algebras in $gr\Vect_\kk$. We have a morphism $gr:\Aut_W(V)\to \Aut(Gr^W V)$.

Let  $\kk\to R$ be a commutative $\kk$-algebra. The
correspondence
$$R \mapsto \AUT_W{(V)}(R) := \Aut_W(V \otimes_{\kk}R)$$
defines a functor
$\AUT_W(V):\cat{Alg}_\kk \lra \cat{Gps}$
from the category  $\cat{Alg}_\kk$  of commutative $\kk$-algebras, to the category $\cat{Gps}$ of groups. Clearly, we have $\AUT_W{(V)} (\kk) = \Aut_W(V)$. We define in a similar fashion a functor $\AUT(Gr^WV)$ from $\cat{Alg}_\kk$ to $\cat{Gps}$.
 
We recall the following properties:

\begin{prop}\label{algebraicgroups}
Let $(V,W)$ be as above.
\begin{enumerate}
\item $\AUT_W(V)$
is a group scheme whose group of $\kk$-points is $\Aut_W(V)$.
\item The functor $Gr^W$ induces a morphism $\underline{gr} :\AUT_W{(V)} \to \AUT(Gr^W V)$  of group schemes.
\item The kernel
$U := \Ker \left( \underline{gr} :\AUT_W(V) \to
\AUT(Gr^W V) \right)$
is a unipotent group scheme over $\kk$.
\end{enumerate}
\end{prop}

\begin{proof} 
We first observe that there is an isomorphism
\[\AUT_W(V)\cong\on{lim}_S\AUT_W(V_S)\]
in which the limit is taken over the poset of finite sets $S$ of objects of $\Cc$ and $V_S$ denotes the restriction of $V$ to those objects. We can write the groups $\Aut_W(V)$ and $\AUT(Gr^WV)$ as similar limits. 
Therefore we may restrict to the case when $\Cc$ has finitely many objects and
prove that in this case, the above objects live in the category of algebraic groups.

Let $N$ be such that the vector space $\oplus_{c\in\Cc} V(c)$ can be linearly embedded in $\kk^N$. Then $\Aut_W(V)$ is the closed subgroup of $\mathrm{GL}_N(\kk)$ defined by the polynomial equations that express the data of a filtration preserving $\Cc$-algebra automorphism.
Similarly, inside the functor of linear automorphisms $\oplus_{c\in\Cc} V(c)\otimes_\kk R\lra \oplus_{c\in\Cc} V(c)\otimes_\kk R$, let $F(R)$ 
be those preserving the structure of $V$ as a $\Cc$-algebra in filtered vector spaces.
The condition of preserving the filtration and the algebra structure is given by polynomial
equations in the matrix entries and so $F$ is representable (this is also explained in Section 7.6 of \cite{Waterhouse}). It follows that 
$\AUT_W(V)$ is an algebraic group and its group of $\kk$-points is $\Aut_W(V)$. 
Hence (1) is satisfied.

For every commutative $\kk$-algebra $R$, the map
\[
\AUT_W{(V)}(R) = \Aut_W(V\otimes_{\kk}R) \longrightarrow \Aut(Gr^W (V\otimes_{\kk}R)) = \AUT(Gr^W V)(R)
\]
is a morphism of groups which is natural in $R$. Thus (2)
follows and hence the kernel $U$ is an algebraic group.
It now suffices to take a basis of $\oplus_{c\in\Cc} V(c)$ compatible with $W$. Then we may view $U$ as
a subgroup of the group of upper-triangular matrices with 1's on the diagonal. Hence (3) is satisfied.
\end{proof}

\begin{lemm}
\label{split_autos}
Let $(V,W)$ be as above. The following assertions are equivalent:
\begin{enumerate}
\item The pair $(V,W)$ admits a lax symmetric monoidal splitting: $W_pV\cong\bigoplus_{q\leq p} Gr^W_qV$.
\item The morphism $gr:\Aut_W(V)\to \Aut(Gr^W V)$ is surjective.
\item There exists $\alpha\in\kk^*$ which is not a root of unity together with an automorphism $\Phi\in \Aut_W(V)$
such that $gr(\Phi)=\psi_\alpha$ is the grading automorphism of $Gr^W V$ associated with $\alpha$, defined by 
$$\psi_\alpha(a)=\alpha^{p}a,\text{ for }a\in Gr_p^WV.$$
\end{enumerate}
\end{lemm}

\begin{proof}
 The implications $(1)\Rightarrow (2)\Rightarrow (3)$ are trivial.
We show that (3) implies (1). Let $\Phi\in \Aut_W(V)$ be such that $gr \Phi=\psi_\alpha$. We will first produce a decomposition $\Phi=\Phi_s.\Phi_u$ which is such that for any object $c$ of $\Cc$, the restrictions $(\Phi_s(c),\Phi_u(c))$ is a Jordan decomposition for $\Phi(c)$. In order to do that, recall that we have an isomorphism
\[\Aut_W(V)=\on{lim}_{S}\Aut_W(V_S)\]
where the limit is taken over the poset of finite subsets $S$ of objects of $\Cc$ and $V_S$ denotes the restriction of $V$ to the subset $S$. For each of the groups $\Aut_W(V_S)$ we can find a Jordan decomposition of the image of $\Phi$ in each of them. The transition maps between those groups preserve this decomposition and it follows that this decomposition induces a decomposition of $\Phi$ with the desired property.

By \cite[Theorem 4.4]{Borel}, there is a decomposition of the form $V(c)=V'(c)\oplus V''(c)$,
where
$$V'(c)=\bigoplus V_{p}(c)\text{ with }V_{p}(c):=\Ker(\Phi_s(c)-\alpha^{p}I)$$
and $V''(c)$ is the complementary subspace corresponding to the remaining factors of the characteristic polynomial of $\Phi_s(c)$. 
By assumption, $Gr^WV(c)$ contains nothing but eigenspaces of eigenvalue $\alpha^p$. Therefore we have $Gr^WV''(c)=0$
and
one concludes that $V''(c)=0$.

In order to show that $W_pV=\bigoplus_{i\leq p}V_p$ it suffices to prove it objectwise. 
Let $c$ be an object of $\Cc$. For $x\in V_p(c)$, let $q$ be the smallest integer such that $x\in W_qV(c)$.
Then $x$ defines a class $x+W_{q+1}V(c)\in gr V(c)$, and
$$\psi_\alpha(x+W_{q-1}V(c))=\alpha^{q}x+W_{q-1}V(c)=\Phi(x)+W_{q-1}V(c)=\alpha^{p}x+W_{q-1}V(c).$$
Then $(\alpha^{q}-\alpha^{p})x\in W_{q-1}V(c)$. Since $x\notin W_{q-1}V(c)$ we have $q=p$, hence $x\in W_pV$.
\end{proof}

We may now state and prove the main theorem of this section.

\begin{theo}\label{descens_rsplittings}
Let $(V,W)$ be a map of colored operads $\Cc\lra \Ff\Vect_\kk$ such that for each color $c$ of $\Cc$, the vector space $V(c)$ is finite dimensional. Let $\kk\subset\mathbb{K}$ be a field extension. Then $V$ admits a lax symmetric monoidal splitting if and only if $V_{\mathbb{K}}:=V\otimes_\kk\mathbb{K}:\Cc\lra \Vect_\mathbb{K}$ admits a lax symmetric monoidal splitting.
\end{theo}

\begin{proof}
We may assume without loss of generality that $\mathbb{K}$ is algebraically closed. If $V_{\mathbb{K}}$ admits a splitting, 
the map 
$$
\AUT_W(V)(\mathbb{K}) \longrightarrow \AUT(Gr^WV)(\mathbb{K})
$$
is surjective by Lemma $\ref{split_autos}$. Our goal is to prove surjectivity of 
$$
\AUT_W(V)(\kk) \longrightarrow \AUT(Gr^WV)(\kk).
$$
As in Proposition \ref{algebraicgroups}, we can write those groups as filtered limits. Since an inverse limit of surjections is a surjection, it is enough to prove the result when $\Cc$ has finitely many objects.

From \cite[Section 18.1]{Waterhouse} there is an exact sequence of groups
$$
1 \longrightarrow U(\mathbf{\kk}) \longrightarrow
\AUT_W(V)(\kk) \longrightarrow \AUT(Gr^WV) (\kk)
\longrightarrow H^1(\mathbb{K}/\kk, U)
\longrightarrow \dots
$$
where $U$ is a unipotent algebraic group by Proposition $\ref{algebraicgroups}$ and our assumption that $\Cc$ has finitely many objects. Since $\mathbf{\kk}$ has characteristic zero the group
$H^1(\mathbb{K}/\kk, U)$ is trivial (see 
\cite[Example 18.2.e]{Waterhouse}) and we deduce the desired surjectivity.
\end{proof}

From this theorem we deduce that Deligne's splitting holds over $\Q$. We record this fact in the following Lemma.

\begin{lemm}[Deligne's splitting over $\Q$]\label{Deligne_splitting_over_Q}
The forgetful functor $\Pi_\Q^{W}:\MHS_\Q\lra \Ff\Vect_\Q$ given by $(V,W,F)\mapsto (V,W)$ admits a factorization
$$
\xymatrix@R=4pc@C=4pc{
\MHS_\Q\ar[r]^{G}\ar[dr]_{\Pi_\Q^{W}}&gr\Vect_\Q\ar[d]^{U^{fil}}\\
&\Ff\Vect_\Q
}
$$
into lax symmetric monoidal functors. In particular, there is an isomorphism of functors
$$U^{fil}\circ gr\circ \Pi_\Q^{W}\cong \Pi_\Q^{W},$$
where $gr:\Ff\Vect_\Q\lra gr\Vect_\Q$ is the graded functor given by $gr(V_\Q,W)^p=Gr_p^WV_\Q$. 
\end{lemm}
 
\begin{proof}
We apply Theorem \ref{descens_rsplittings} to the lax symmetric monoidal functor $\Pi_\Q^{W}$ using the fact that $\Pi_\Q^{W}\otimes_\Q\C$ admits a splitting by Lemma \ref{Deligne_splitting}.
\end{proof} 

\begin{rem}
We want to emphasize that Theorem \ref{descens_rsplittings} does not say that the splitting of the previous lemma recovers the splitting of Lemma \ref{Deligne_splitting} after tensoring with $\C$. In fact, it can probably be shown that such a splitting cannot exist. Nevertheless, the existence of Deligne's splitting over $\C$ abstractly forces the existence of a similar splitting over $\Q$ which is all this Lemma is saying. Note as well that these are not splittings of mixed Hodge structures, but only of the weight filtration. They are also referred to as \textit{weak splittings} of mixed Hodge structures (see for example \cite[Section 3.1]{PS}).
As is well-known, mixed Hodge structures do not split in general.
\end{rem}

The above splitting over $\Q$ yields the following ``purity implies formality'' statement in the abstract setting of functors defined from the category of complexes of mixed Hodge structures. Given a rational number $\alpha$,
denote by $\ch(\MHS_\Q)^{\alpha\text{-}pure}$ the full subcategory of $\ch(\MHS_\Q)$ of complexes with pure weight $\alpha$ homology: an object $(K,W,F)$ in $\ch(\MHS_\Q)^{\alpha\text{-}pure}$ is such that
$Gr^p_WH_n(K)=0$ for all $p\neq \alpha n.$

\begin{coro}\label{coro-purity formality Q}
The restriction of the functor $\Pi_\Q:\ch(\MHS_\Q)\lra \ch(\Q)$ to
the category  
$\ch(\MHS_\Q)^{\alpha\text{-}pure}$
is formal for any non-zero rational number $\alpha$.
\end{coro}

\begin{proof}
This follows from Proposition \ref{decomposition_formal} together with Lemma \ref{Deligne_splitting_over_Q}.
\end{proof}

\section{Mixed Hodge complexes}\label{Section_MHC}
In this section, we construct an equivalence of symmetric monoidal $\infty$-categories between mixed Hodge complexes and complexes of mixed Hodge structures, lifting Beilinson's equivalence of triangulated categories.

We first recall the notion of mixed Hodge complex introduced by Deligne in \cite{DeHIII}
in its chain complex version (with homological degree).
Note that, in contrast with the classical setting of mixed Hodge theory,
in the homological version of a mixed Hodge complex, the weight filtration $W$ will be decreasing while the Hodge filtration $F$ will be increasing.

Let $\kk\subset \R$ be a subfield of the real numbers.

\begin{defi}\label{defMHC}
A \textit{mixed Hodge complex over $\kk$} is given by a filtered chain complex $(K_{\kk},W)$ over $\kk$, a bifiltered chain complex $(K_{\C},W,F)$ over $\C$, together with a finite string of filtered quasi-isomorphisms of filtered complexes of $\C$-vector spaces
$$(K_\kk,W)\otimes\C\xrightarrow{\alpha_1}(K_1,W)\xleftarrow{\alpha_2}\cdots\xrightarrow{\alpha_{l-1}} (K_{l-1},W)\xrightarrow{\alpha_l} (K_\C,W).$$
We call $l$ the \textit{length} of the mixed Hodge complex.
The following axioms must be satisfied:
\begin{enumerate}
\item[($\mathrm{MH}_0$)] The homology $H_*(K_\kk)$ is bounded and finite-dimensional.
\item[($\mathrm{MH}_1$)] The differential of $Gr^p_WK_\C$ is strictly compatible with $F$.
\item[($\mathrm{MH}_2$)] The filtration on $H_n(Gr_W^pK_{\C})$ induced by $F$ makes $H_n(Gr_W^pK_\kk)$ into a pure Hodge structure of
weight $p+n$.
\end{enumerate}

Such a mixed Hodge complex will be denoted by $\Kk=\{(K_\kk,W),(K_\C,W,F)\}$, omitting the data of the comparison morphisms $\alpha_i$.
\end{defi}

The above axioms imply that for all $n\geq 0$ the triple $(H_n(K_\kk),W[n],F)$ is a mixed Hodge structure over $\kk$, 
where $W[n]$ is the shifted weight filtration given by 
\[W[n]^p H_n(K_\kk):=W^{p-n}H_n(K_\kk).\]
Morphisms of mixed Hodge complexes are given by levelwise bifiltered morphisms of complexes making the corresponding diagrams commute. Denote by $\MHC_\kk$ the category of mixed Hodge complexes 
of a certain fixed length, which we omit in the notation.
The tensor product of mixed Hodge complexes is again a mixed Hodge complex (see \cite[Lemma 3.20]{PS}). 
This makes $\MHC_\kk$ into a symmetric monoidal category, with a filtered variant of the K\"{u}nneth formula.

\begin{defi}
A morphism $f:K\to L$ in $\MHC_\kk$ is said to be a \textit{weak equivalence} if $H_*(f_\kk)$ is an isomorphism of $\kk$-vector spaces.
\end{defi}

Since the category of mixed Hodge structures is abelian, the
homology of every complex of mixed Hodge structures is a graded mixed Hodge structure. We have a functor 
$$\Tt:\chb(\MHS_\kk)\lra \MHC_\kk$$
given on objects by 
$(K,W,F)\mapsto \{(K,TW),(K\otimes\C,TW,F)\}$,
where $TW$ is the shifted filtration $(TW)^pK_n:=W^{p+n}K_n$.
The comparison morphisms $\alpha_i$ are the identity.
Also, $\Tt$ is the identity on morphisms. This functor clearly preserves weak equivalences.
\begin{lemm}
The shift functor $\Tt:\chb(\MHS_\kk)\lra \MHC_\kk$ is strong symmetric monoidal.
\end{lemm}

\begin{proof}
It suffices to note that given filtered complexes $(K,W)$ and $(L,W)$, we have: 
\[T(W\otimes W)^p(K\otimes L)_n=(TW\otimes TW)^p(K\otimes L)_n.\qedhere\]
\end{proof}

Beilinson \cite{Be} gave an equivalence of categories between the derived category  of mixed Hodge structures and the homotopy category of a shifted version of mixed Hodge complexes. We will require a finer version of Beilinson's equivalence, in terms of symmetric monoidal $\infty$-categories.
Denote by $\iMHC_\kk$ the $\infty$-category obtained
by inverting weak equivalences of mixed Hodge complexes, omitting the length in the notation. As explained in \cite[Theorem 2.7.]{Drew}, this object is canonically a symmetric monoidal stable $\infty$-category. Note that in loc. cit., mixed Hodge complexes have fixed length 2 and are polarized. The results of \cite{Drew} as well as Beilinson's equivalence, are equally valid for the category of 
mixed Hodge complexes of an arbitrary fixed length.

\begin{theo}\label{theo: equivinfty}
The shift functor induces an equivalence 
$\ich^b(\MHS_\kk)\lra \iMHC_\kk$ of symmetric monoidal $\infty$-categories.
\end{theo}

\begin{proof}
A proof in the polarizable setting appears in \cite{Drew}. Also, in \cite{BNT}, a similar statement 
is proven for a shifted version of mixed Hodge complexes. 
We sketch a proof in our setting. 

We first observe as in Lemma 2.6 of \cite{BNT} that both $\infty$-categories are stable and that the shift functor is exact. The stability of $\iMHC_\kk$ follows from the observation that this $\infty$-category is the Verdier quotient at the acyclic complexes of the $\infty$-category of mixed Hodge complexes with the homotopy equivalences inverted. This last $\infty$-category underlies a dg-category that can easily be seen to be stable. The stability of $\ich^b(\MHS_\kk)$ follows in a similar way. Since a complex of mixed Hodge structures is acyclic if and only if the underlying complex of $\kk$-vector spaces is acyclic, and $\Tt$ is the identity on the underlying complexes of $\kk$-vector spaces, it follows that $\Tt$ is exact. Therefore, in order to prove that $\Tt$ is an equivalence of $\infty$-categories, it suffices to show that it induces an equivalence of homotopy categories
\[D^{b}(\MHS_\kk)\lra \ho(\MHC_\kk).\]

In \cite[ Lemma 3.11]{Be}, it is proven that the shift functor $\Tt:\ch^b(\MHS_\kk^{p})\lra \MHC_\kk^p$ induces an equivalence at the level of homotopy categories. Here the superindex $p$ indicates that the mixed Hodge objects are polarized. One may verify that Beilinson's proof is equally valid if we remove the polarization (see also \cite[Theorem 4.10]{CG2} 
where Beilinson's equivalence is proven in the non-polarized version via other methods). The fact that $\Tt$ can be given the structure of a strong symmetric monoidal $\infty$-functor follows from the work of Drew in \cite{Drew}.
\end{proof}

\section{Logarithmic de Rham currents}\label{Section_logaritmic}

The goal of this section is to construct a strong symmetric monoidal $\infty$-functor from algebraic varieties over $\C$ to mixed Hodge complexes which computes the correct mixed Hodge structure after passing to homology. The construction for smooth varieties is relatively straightforward. It suffices to take a functorial mixed Hodge complex model for the cochains as constructed for instance in \cite{Na} and dualize it. The monoidality of that functor is slightly tricky as one has to move to the world of $\infty$-categories for it to exist. Once one has constructed this functor for smooth varieties, it can be extended to more general varieties by standard descent arguments.

We denote by $\Var_{\C}$ the category of complex schemes that are reduced, separated and of finite type. We use the word variety for an object of this category. We denote by $\Sm_{\C}$ the subcategory of smooth schemes. Both of these categories are essentially small (i.e. there is a set of isomorphisms classes of objects) and symmetric monoidal under the cartesian product.

We will make use of the following very simple observation.

\begin{prop}\label{prop: cartesian product oplax}
Let $\Cc$ and $\Dd$ be two categories with finite products seen as symmetric monoidal categories with respect to the product. Then any functor $F:\Cc\lra \Dd$ has a preferred oplax symmetric monoidal structure.
\end{prop}

\begin{proof}
We need to construct comparison morphisms $F(c\times c')\lra F(c)\times F(c')$. By definition of the product, there is a unique such functor whose composition with the first projection is the map $F(c\times c')\lra F(c)$ induced by the first projection $c\times c'\lra c$ and whose composition with the second projection is the map $F(c\times c')\lra F(c')$ induced by the second projection $c\times c'\lra c'$. Similarly, one has a unique map $F(\ast)\lra \ast$. One checks easily that these two maps give $F$ the structure of an oplax symmetric monoidal functor.
\end{proof}

\subsection{For smooth varieties}

In this section, we construct a lax symmetric monoidal functor 
\[\D_*:\inn(\cat{Sm}_{\mathbb{C}})\lra \iMHC_\Q\] such that its composition with the forgetful functor
$\iMHC_\Q\lra \ich(\Q)$
is naturally weakly equivalent to $S_*(-,\Q)$ as a lax symmetric monoidal functor. 

We will adapt Navarro-Aznar's construction of mixed Hodge diagrams \cite{Na}.
Let $X$ be a smooth projective complex variety and $j:U\hookrightarrow X$ an open subvariety such that $D:=X-U$ is a normal crossing divisor.
Denote by $\Aa^*_X$ the analytic de Rham complex
of the underlying real analytic variety of $X$ and let $\Aa^*_X(\log D)$ denote the subsheaf of $j_*\Aa^*_U$  of logarithmic forms in $D$.
Note that in Deligne's approach to mixed Hodge theory, the sheaf 
$\Omega_X^*(\log D)$ of holomorphic forms on $X$ with logarithmic poles along $D$
is used instead. As explained in 8.5 of \cite{Na}, the main 
advantage to consider analytic forms is the natural real structure obtained, together with a decomposition of the form
\[\Aa^n_X(\log D)\otimes \C=\bigoplus_{p+q=n} \Aa_X^{p,q}(\log D).\]
Also, there is an inclusion $\Omega_X^*(\log D)\hookrightarrow \Aa^*_X(\log D)\otimes\C$ which is a quasi-isomorphism
and $\Aa^*_X(\log D)$
may be naturally endowed with a multiplicative weight filtration $W$ (see 8.3 of \cite{Na}). Proposition 8.4 of loc. cit. gives a string of quasi-isomorphisms of sheaves of filtered cdga's over $\R$:
\[
(\mathbf{R}_{\TW} j_*\underline{\Q}_U,\tau)\otimes\R
\xrightarrow{\sim}
(\mathbf{R}_{\TW} j_*\Aa^*_U,\tau)
\xleftarrow{\sim}
(\Aa^*_X(\text{log} D),\tau)
\xrightarrow{\sim}(\Aa^*_X(\text{log} D),W),
\]
where $\tau$ is the canonical filtration.
In this diagram, 
\[\mathbf{R}_{\TW}j_*:\mathrm{Sh}(U,\ch^{\leq 0}(\kk))\lra \mathrm{Sh}(X,\ch^{\leq 0}(\kk))\]
is the functor defined by 
\[\mathbf{R}_{\TW}j_*:=\mathbf{s}_{\TW}\circ j_*\circ G^+\]
where $$G^+:\mathrm{Sh}(X,\ch^{\leq 0}(\kk))\lra \Delta \mathrm{Sh}(X,\ch^{\leq 0}(\kk))$$
is the Godement canonical cosimplicial resolution functor and
\[\mathbf{s}_{\TW}:\Delta \mathrm{Sh}(X,\ch^{\leq 0}(\kk))\lra \mathrm{Sh}(X,\ch^{\leq 0}(\kk))\]
is the Thom-Whitney simple functor introduced by Navarro in Section 2 of loc. cit. Both functors are lax symmetric monoidal and hence $\mathbf{R}_{\TW}j_*$ is a lax symmetric monoidal functor (see \cite[Section 3.2]{RR}). 

The complex
$\Aa^*_X(\log D)\otimes\C$ carries a natural multiplicative Hodge filtration
$F$ (see 8.6 of \cite{Na}). 
The above string of quasi-isomorphisms gives a commutative algebra object in
(cohomological) mixed Hodge complexes after taking global sections. Specifically,
the composition 
\[\mathbf{R}_{\TW}\Gamma(X,-):=\mathbf{s}_{\TW}\circ \Gamma(X,-)\circ G^+\]
gives a derived global sections functor
\[\mathbf{R}_{\TW}\Gamma(X,-):\mathrm{Sh}(X,\ch^{\leq 0}(\Q))\lra \ch^{\leq 0}(\Q)\]
which again is lax symmetric monoidal. There is also a filtered version of
this functor defined via the filtered Thom-Whitney simple (see Section 6 of \cite{Na}). Theorem 8.15 of loc. cit. asserts that by applying
the (bi)filtered versions of $\mathbf{R}_{\TW}\Gamma(X,-)$ to each of the pieces of the above string of quasi-isomorphisms, one obtains a commutative algebra object in mixed Hodge complexes $\Hh dg(X,U)$ whose cohomology gives Deligne's mixed Hodge structure on $H^*(U,\Q)$ and such that 
$$\Hh dg(X,U)_\Q=\mathbf{R}_{\TW}\Gamma(X,\mathbf{R}_{\TW}j_*\underline{\Q}_U)$$
is naturally quasi-isomorphic to $S^*(U,\C)$ (as a cochain complex). 
This construction is functorial for morphisms of pairs $f:(X,U)\to (X',U')$. The 
definition of $\Hh dg(f)$ follows as in the additive setting 
(see \cite[Lemma 6.1.2]{Huber} for details), by replacing the classical additive total simple
functor with the Thom-Whitney simple functor.

Now we wish to get rid of the dependence on the compactification. For this purpose, we define for $U$ a smooth variety over $\C$, a category $R_U$ whose objects are pairs $(X,U)$ where $X$ is smooth and proper variety containing $U$ as an open subvariety and $X-U$ is a normal crossing divisor. Morphisms in $R_U$ are morphisms of pairs. We then define $\D^*(U)$ by the formula
\[\D^*(U):=\on{colim}_{R_U\op}\Hh dg(X,U)\]
By theorems of Hironaka and Nagata, the category $R_U\op$ is a non-empty filtered category. Note that we have to be slightly careful here as the category of mixed Hodge complexes does not have all filtered colimits. However, we can form this colimit in the category of pairs $(K_\Q,W),(K_\C,W,F)$ having the structure required in Definition \ref{defMHC} but not necessarily satisfying the axioms $\mathrm{MH}_0$, $\mathrm{MH}_1$ and $\mathrm{MH}_2$. Since taking  filtered colimit is an exact functor, we deduce from the classical isomorphism between sheaf cohomology and singular cohomology that there is a quasi-isomorphism
\[\D^*(U)_\Q\to S^*(U,\Q)\]
This shows that the cohomology of $\D^*(U)$ is of finite type and hence, that $\D^*(U)$ satisfies axiom $\mathrm{MH}_0$. The other axioms are similarly easily seen to be satisfied. Moreover, filtered colimits preserve commutative algebra structures, therefore the functor $\D^*$ is a functor from $\Sm_\C\op$ to commutative algebras in $\MHC_\Q$. 

Since the coproduct in commutative algebras is the tensor product, we deduce from the dual of Proposition \ref{prop: cartesian product oplax} that $\D^*$ is canonically a lax symmetric monoidal functor from $\Sm_\C\op$ to $\MHC_\Q$. But since the comparison map
\[\D^*(U)_\Q\otimes_\Q \D^*(V)_\Q\lra \D^*(U\times V)_\Q\]
is a quasi-isomorphism, this functor extends to a strong symmetric monoidal $\infty$-functor 
\[\D^*:\inn(\Sm_\C)\op\lra \iMHC_\Q\]

\begin{rem}
A similar construction for real mixed Hodge complexes is done in \cite[Section 3.1]{bunkeregulators}. There is also a similar construction in \cite{Drew} that includes polarizations.
\end{rem}

Now, the category $\MHC_\Q$ is equipped with a duality functor. It sends a mixed Hodge complex $\{(K_\Q,W),(K_\C,W,F)\}$ to the linear duals $\{(K_\Q^{\vee},W^{\vee}),(K_\C^{\vee},W^{\vee},F^{\vee})\}$ where the dual of a filtered complex is defined as in \ref{Homandtensor}. One checks easily that this dual object satisfies the axioms of a mixed Hodge complex. Moreover, the duality functor $\MHC_\Q\op\lra \MHC_\Q$ is lax symmetric monoidal and preserves weak equivalences of mixed Hodge complexes, therefore it induces a lax symmetric monoidal $\infty$-functor
\[\iMHC_\Q\op\lra\iMHC_\Q\]
but in fact, we have the following proposition.

\begin{prop}\label{prop : duality is monoidal}
The dualization $\infty$-functor
\[\iMHC_\Q\op\lra\iMHC_\Q\]
is strong symmetric monoidal.
\end{prop}

\begin{proof}
It suffices to observe that the canonical map
\[K^{\vee}\otimes L^{\vee}\lra (K\otimes L)^{\vee}\]
is a weak equivalence. This follows from the fact that mixed Hodge complexes are assumed to have finite type cohomology.
\end{proof}

Composing the duality functor with $\D^*$, we get a strong symmetric monoidal $\infty$-functor
\[\D_*:\inn(\Sm_\C)\lra \iMHC_\Q \]

\begin{rem}
One should note that $\D^*$ comes from a lax symmetric monoidal functor from $\Sm_\C\op$ to $\MHC_\Q$. 
On the other hand, $\D_*$ is induced by a strict functor which is neither lax nor oplax. Indeed, it is obtained as the composition of $(\D^*)\op$ which is an oplax symmetric monoidal functor $\Sm_\C\lra(\MHC_\Q)\op$ and the duality functor which is a lax symmetric monoidal functor. Thus, the symmetric monoidal structure on $\D_*$ only exists at the $\infty$-categorical level. 
\end{rem}

To conclude this construction, it remains to compare the functor $\D_*(-)_\Q$ with the singular chains functor. These two functors are naturally quasi-isomorphic as shown in \cite{Na} but we will need to know that they are quasi-isomorphic as symmetric monoidal $\infty$-functors. We denote by $S_*(-,R)$ the singular chain complex functor from the category of topological spaces to the category of chain complexes over a commutative ring $R$. The functor $S_*(-,R)$ is lax symmetric monoidal. Moreover, the natural map 
\[S_*(X,R)\otimes S_*(Y,R)\to S_*(X\times Y,R)\]
is a quasi-isomorphism. This implies that $S_*(-,R)$ induces a strong symmetric monoidal $\infty$-functor from the category of topological spaces to the $\infty$-category $\ich(R)$ of chain complexes over $R$.  We still use the symbol $S_*(-,R)$ to denote this $\infty$-functor.

\begin{theo}\label{theo: equivalence cochains-sheaf}
The functors $\D_*(-)_\Q$ and $S_*(-,\Q)$ are weakly equivalent as strong symmetric monoidal $\infty$-functors from $\inn(\Sm_\C)$ to $\ich(\Q)$. 
\end{theo}

\begin{proof}
We introduce the category $\cat{Man}$ of smooth real manifolds. We consider the $\infty$-category $\icat{PSh}(\cat{Man})$ of presheaves of spaces on the $\infty$-category $\inn(\cat{Man})$. This is a symmetric monoidal $\infty$-category under the product. We can consider the reflective subcategory $\icat{T}$ spanned by presheaves $\Gg$ satisfying the following two conditions:
\begin{enumerate}
\item Given a hypercover $U_\bullet \to M$ of a manifold $M$, the induced map
\[\Gg(M)\to\on{lim}_{\Delta}\Gg(U_\bullet)\]
is an equivalence.
\item For any manifold $M$, the map $\Gg(M)\to \Gg(M\times\mathbb{R})$ induced by the projection $M\times\mathbb{R}\to M$ is an equivalence. 
\end{enumerate}
The presheaves satisfying these conditions are stable under product, hence the $\infty$-category $\icat{T}$ inherits the structure of a symmetric monoidal locally presentable $\infty$-category. It has a universal property that we now describe.

Given another symmetric monoidal locally presentable $\infty$-category $\icat{D}$, we denote by $\Fun^{L,\otimes}(\icat{T},\icat{D})$ the $\infty$-category of colimit preserving strong symmetric monoidal functors $\icat{T}\to\icat{D}$. Then, we can consider the composition
\[\Fun^{L,\otimes}(\icat{T},\icat{D})\to\Fun^{L,\otimes}(\icat{PSh}(\cat{Man}),\icat{D})\to\Fun^{\otimes}(\inn\cat{Man},\icat{D})\]
where the first map is induced by precomposition with the left adjoint to the inclusion $\icat{T}\to \icat{PSh}(\cat{Man})$ and the second map is induced by precomposition with the Yoneda embedding. We claim that the above composition is fully faithful and that its essential image is the full subcategory of $\Fun^{\otimes}(\inn\cat{Man},\icat{D})$ spanned by the functors $F$ that satisfy the following two properties:
\begin{enumerate}
\item Given a hypercover $U_\bullet\to M$ of a manifold $M$, the map
\[\on{colim}_{\Delta\op}F(U_\bullet)\to F(M)\]
is an equivalence.
\item For any manifold $M$, the map $F(M\times\mathbb{R})\to F(M)$ induced by the projection $M\times\mathbb{R}\to M$ is an equivalence. 
\end{enumerate}
This statement can be deduced from the theory of localizations of symmetric monoidal $\infty$-categories (see \cite[Section 3]{hinichdwyer}).

In particular, there exists an essentially unique strong symmetric monoidal and colimit preserving functor from $\icat{T}$ to $\icat{S}$ (the $\infty$-category of spaces) that is determined by the fact that it sends a manifold $M$ to the simplicial set $\on{Sing}(M)$. This functor is an equivalence of $\infty$-categories. This is a floklore result. A proof of a model category version of this fact can be found in \cite[Proposition 8.3.]{duggeruniversal}.

The $\infty$-category $\icat{S}$ is the unit of the symmetric monoidal $\infty$-category of presentable $\infty$-categories. It follows that it has a commutative algebra structure (which corresponds to the symmetric monoidal structure coming from the cartesian product) and that it is the initial symmetric monoidal presentable $\infty$-category. Since $\icat{T}$ is equivalent to $\icat{S}$ as a symmetric monoidal presentable $\infty$-category, we deduce that, up to equivalence, there is a unique functor $\icat{T}\lra\ich(\Q)$ that is strong symmetric monoidal and colimit preserving. But, using the universal property of $\icat{T}$, we easily see that $S_*(-,\Q)$ and $\D_*(-)_\Q$ can be extended to strong symmetric monoidal and colimits preserving functors from $\icat{T}$ to $\ich(\Q)$. It follows that they must be equivalent.
\end{proof}

\subsection{For varieties}

In this subsection, we extend the construction of the previous subsection to the category of varieties.

We have the site $(\Var_{\C})_{pro}$ of varieties over $\C$ with  the proper topology and the site $(\Sm_\C)_{pro}$ which is the restriction of this site to the category of smooth varieties (see \cite[Section 3.5]{blanctopological} for the definition of the proper topology).

\begin{prop}[Blanc]\label{prop:Blanc}
Let $\icat{C}$ be a symmetric monoidal presentable $\infty$-category. We denote by $\Fun^{\otimes}_{pro,}(\Var_\C,\icat{C})$ the $\infty$-category of strong symmetric monoidal functors from $\Var_\C$ to $\icat{C}$ whose underlying functor satisfies descent with respect to proper hypercovers. Similarly, we denote by $\Fun^{\otimes}_{pro}(\Sm_\C,\icat{C})$ the $\infty$-category of strong symmetric monoidal functors from $\Sm_\C$ to $\icat{C}$ whose underlying functor satisfies descent with respect to proper hypercovers. The restriction functor
\[\Fun^{\otimes}_{pro}(\Var_\C,\icat{C})\lra \Fun^{\otimes}_{pro}(\Sm_\C,\icat{C}) \]
is an equivalence.
\end{prop}

\begin{proof}
We have the categories $\Fun(\Var_\C\op,\cat{sSet})$ and $\Fun(\Sm_\C\op,\cat{sSet})$ of presheaves of simplicial sets over $\Var_{\C}$ and $\Sm_\C$ respectively. These categories are related by an adjunction
\[\pi^*:\Fun(\Sm_\C\op,\cat{sSet})\leftrightarrows \Fun(\Var_\C\op,\cat{sSet}):\pi_*\]
where the right adjoint $\pi_*$ is just the restriction of a presheaf to smooth varieties. Both sides of this adjunction have a symmetric monoidal structure by taking objectwise product. The functor $\pi_*$ is obviously strong symmetric monoidal. We can equip both sides with the local model structure with respect to the proper topology. We obtain a Quillen adjunction
\[\pi^*:\Fun_{pro}(\Sm_\C\op,\cat{sSet})\leftrightarrows \Fun_{pro}(\Var_\C\op,\cat{sSet}):\pi_*\]
between symmetric monoidal model categories in which the right adjoint is a strong symmetric monoidal functor. In \cite[Proposition 3.22]{blanctopological}, it is proved that this is a Quillen equivalence. The model category $\Fun_{pro}(\Sm_\C\op,\cat{sSet})$ presents the $\infty$-topos of hypercomplete sheaves over the proper site on $\Sm_\C$ and similarly for model category $\Fun_{pro}(\Var_\C\op,\cat{sSet})$. Therefore, this Quillen equivalence implies that these two $\infty$-topoi are equivalent. Moreover, as in the proof of \ref{theo: equivalence cochains-sheaf}, these topoi, seen as symmetric monoidal presentable $\infty$-categories under the cartesian product, represent the functor  $\icat{C}\mapsto \Fun^{\otimes}_{pro}(\Sm_\C,\icat{C})$ (resp. $\icat{C}\mapsto \Fun^{\otimes}_{pro}(\Var_\C,\icat{C})$). The result immediately follows.
\end{proof}

\begin{theo}
Up to weak equivalences, there is a unique strong symmetric monoidal functor
\[\D_*:\inn(\Var_\C)\lra \iMHC_\Q\]
which satisfies descent with respect to proper hypercovers and whose restriction to $\Sm_\C$ is equivalent to the functor $\D_*$ constructed in the previous subsection.

There is also a unique strong symmetric monoidal functor
\[\D^*:\inn(\Var_\C)\op\lra \iMHC_\Q\]
which satisfies descent with respect to proper hypercovers and whose restriction to $\Sm_\C$ is equivalent to the functor $\D^*$ constructed in the previous subsection.
\end{theo}

\begin{proof}
Let $\on{Ind}(\iMHC_\Q)$ be the Ind-category of the $\infty$-category of mixed Hodge complexes. This is a stable presentable $\infty$-category. We first prove that the composite
\[\D_*:\inn(\Sm_\C)\lra \iMHC_\Q\lra \on{Ind}(\iMHC_\Q)\]
satisfies descent with respect to proper hypercovers. Let $Y$ be a smooth variety and $X_\bullet\to Y$ be a hypercover for the proper topology. We wish to prove that the map
\[\alpha:\on{colim}_{\Delta\op} \D_*(X_\bullet)\lra \D_*(Y)\]
is an equivalence in $\on{Ind}(\iMHC_\Q)$. By \cite[Proposition 3.24]{blanctopological} and the fact that taking singular chains commutes with homotopy colimits in spaces, we see that the map
\[\beta:\on{colim}_{\Delta\op}S_*(X_\bullet,\Q)\lra S_*(Y,\Q)\]
is an equivalence. On the other hand, writing $\ich(\Q)^{\omega}$ for the $\infty$-category of chain complexes whose homology is finite dimensional, the forgetful functor
\[U:\on{Ind}(\iMHC_\Q)\lra \on{Ind}(\ich(\Q)^{\omega})\simeq \ich(\Q)\]
preserves colimits and by Theorem \ref{theo: equivalence cochains-sheaf}, the composite $U\circ\D_*$ is weakly equivalent to $S_*(-,\Q)$. Therefore, the map $\beta$ is weakly equivalent to the map $U(\alpha)$ in particular, we deduce that the source of $\alpha$ is in $\iMHC_\Q$ (as opposed to $\on{Ind}(\iMHC_\Q)$). 
And since the functor $U:\iMHC_\Q\to \ich(\C)$ is conservative, it follows that $\alpha$ is an equivalence as desired. 

Hence, by Proposition \ref{prop:Blanc}, there is a unique extension of $\D_*$ to a strong symmetric monoidal functor $\inn(\Var_\C)\lra\on{Ind}(\iMHC_\Q)$ that has proper descent. Moreover, by the first paragraph of this proof, if $Y$ is an object of $\Var_\C$ and $X_\bullet\lra Y$ is a proper hypercover by smooth varieties, then $\on{colim}_{\Delta\op}\D_*(X_\bullet,\Q)$ has finitely generated homology. It follows that this unique extension of $\D_*$ to $\Var_\C$ lands in $\iMHC_\Q\subset \on{Ind}(\iMHC_\Q)$.

For the case of $\D^*$, we know from Proposition \ref{prop : duality is monoidal} that dualization induces a strong symmetric monoidal equivalence of $\infty$-categories $\iMHC_\Q\op\simeq \iMHC_\Q$ (we emphasize that, as a functor, dualization is only lax symmetric monoidal but as an $\infty$-functor it is strong symmetric monoidal). Thus, we see that we have no other choice but to define $\D^*$ as the composite
\[\inn(\Var)\op \xrightarrow{(\D_*)\op}\iMHC_\Q\op\xrightarrow{(-)^{\vee}}\iMHC_\Q\]
and this will be the unique strong symmetric monoidal functor 
\[\D^*:\inn(\Var_\C)\op\lra \iMHC_\Q\]
which satisfies descent with respect to proper hypercovers and whose restriction to $\Sm_\C$ is equivalent to the functor $\D^*$ constructed in the previous subsection.
\end{proof}

\begin{prop}\label{prop:equivalence D S}
\begin{enumerate}
\item There is a weak equivalence $\D_*(-)_\Q\simeq S_*(-,\Q)$ in the category of strong symmetric monoidal $\infty$-functors $\inn(\Var_\C)\lra \ich(\Q)$. 
\item There is a weak equivalence $\Aa_{PL}^*(-)\simeq\D^*(-)_\Q\simeq S^*(-,\Q)$ in the $\infty$-category of strong symmetric monoidal $\infty$-functors $\inn(\Var_\C)\op\lra \ich(\Q)$.
\end{enumerate}
\end{prop}

\begin{proof}
We prove the first claim. By construction $\D_*(-)_\Q$ is a symmetric monoidal functor that satisfies proper descent. By \cite[Proposition 3.24]{blanctopological}, the same is true for $S_*(-,\Q)$. Since these two functors are moreover weakly equivalent when restricted to $\Sm_\C$, they are equivalent by Proposition \ref{prop:Blanc}. 

The linear dual functor is  strong symmetric monoidal when restricted to chain complexes whose homology is of finite type. Moreover, both $S_*(-,\Q)$ and $\D_*(-)_\Q$ land in the $\infty$-category of such chain complexes. Therefore, the equivalence $S^*(-,\Q)\simeq \D^*(-)_{\Q}$ follows from the first part. The equivalence $\Aa_{PL}^*(-)\simeq S^*(-,\Q)$ is classical.
\end{proof}

\section{Formality of the singular chains functor}\label{Section_main}

In this section, we prove the main results of the paper on the formality of the singular chains functor. 
We also explain some applications to operad formality.

\begin{defi}
Let $X$ be a complex variety and let $\alpha$ be a rational number. We say that the weight filtration on $H^*(X,\Q)$ is \textit{$\alpha$-pure}
if for all $n\geq 0$ we have
$$Gr^W_pH^n(X,\Q)=0\text{ for all }p\neq \alpha n.$$
\end{defi}

\begin{rem}
Note that since the weight filtration on $H^n(-,\Q)$ has weights in the interval $[0,2n]\cap \Z$, the above definition makes 
sense only for $\alpha\in[0, 2]\cap \Q$.
For $\alpha=1$ we recover the purity property shared by the cohomology of smooth projective varieties.
A very simple example of a variety whose filtration is $\alpha$-pure, with $\alpha$ not integer,
is given by $\C^2\setminus\{0\}$. Its reduced cohomology is concentrated in degree $3$ and weight $4$,
so its weight filtration is $4/3$-pure. 
We refer to Proposition $\ref{prop: purity for subspace arrangements}$ in the following section for 
more elaborate examples.
\end{rem}

Here is our main theorem.

\begin{theo}\label{theo: main covariant}
Let $\alpha$ be a non-zero rational number. The singular chains functor \[S_*(-,\Q):\Var_\C\lra \ch(\Q)\]
is formal as a lax symmetric monoidal functor when restricted to varieties whose weight filtration in cohomology is $\alpha$-pure.
\end{theo}

\begin{proof}
By Corollary \ref{coro: Hinich rigidification}, it suffices to prove that this functor is formal as an $\infty$-lax symmetric monoidal functor. By Proposition \ref{prop:equivalence D S}, it is equivalent to prove that $\D_*(-)_{\Q}$ is formal. We denote by $\bar{\D}_*$ the composite of $\D_*$ with a strong symmetric monoidal inverse of the equivalence of Theorem \ref{theo: equivinfty}. Because of that theorem, $\D_*(-)_\Q$ is weakly equivalent to $\Pi_\Q\circ \bar{\D}_*$. The restriction of  $\bar{\D}_*$  to $\Var_\C^{\alpha\text{-}pure}$ lands in $\ich(\MHS_\Q)^{\alpha\text{-}pure}$, the full subcategory of $\ich(\MHS_\Q)$ spanned by chain complexes whose homology is $\alpha$-pure. 
By Corollary \ref{coro-purity formality Q}, the $\infty$-functor $\Pi_\Q$ from $\ich(\MHS_\Q)^{\alpha\text{-}pure}$ to $\ich(\Q)$ is formal and hence so is $\Pi_\Q\circ \bar{\D}_*$.
\end{proof}

We now list a few applications of this result.

\subsection{Noncommutative little disks operad}\label{subsection: ncld}

The authors of \cite{dotsenkononcommutative} introduce two nonsymmetric topological operads
$\mathcal{A}s_{S^1}$ and $\mathcal{A}s_{S^1}\rtimes S^1$.
In each arity, these operads are given by a product of copies of $\C-\{0\}$ and the operad maps can be checked to be algebraic maps. It follows that the operads $\mathcal{A}s_{S^1}$ and $\mathcal{A}s_{S^1}\rtimes S^1$ are operads in the category $\on{Sm}_\C$ and that the
weight filtration on their cohomology is $2$-pure. Therefore, by \ref{theo: main covariant} we have the following result.

\begin{theo}\label{theo: little}
The operads $S_*(\mathcal{A}s_{S^1},\Q)$ and $S_*(\mathcal{A}s_{S^1}\rtimes S^1,\Q)$ are formal.
\end{theo}

\begin{rem}
The fact that the operad $S_*(\mathcal{A}s_{S^1},\Q)$ is formal is proved in \cite[Proposition 7]{dotsenkononcommutative} by a more elementary method and it is true even with integral coefficients. The other formality result was however unknown to the authors of \cite{dotsenkononcommutative}.
\end{rem}

\subsection{Self-maps of the projective line}\label{subsection : self maps}

We denote by $F_d$ the algebraic variety of degree $d$ algebraic maps from $\mathbf{P}^1_{\C}$ to itself that send the point $\infty$ to the point $1$. Explicitly, a point in $F_d$ is a pair $(f,g)$ of degree $d$ monic polynomials without any common roots. Sending a monic polynomial to its set of coefficients, we may see the variety $F_d$ as a Zariski open subset of $\mathbf{A}^{2d}_{\C}$. See \cite[Section 5]{horelmotivic} for more details.

\begin{prop}The weight filtration on $H^*(F_d,\Q)$ is $2$-pure. 
\end{prop}

\begin{proof}
The variety $F_d$ is denoted $\on{Poly}_1^{d,2}$ in \cite[Definition 1.1.]{farbtopology}. It is explained in Step 4 of the proof of Theorem 1.2 in that paper, that the variety $F_d$ is the quotient of the complement of a hyperplane arrangement $H$ in $\mathbf{A}^{2d}_{\C}$ by the group $\Sigma_d\times\Sigma_d$ acting by permuting the coordinates. The quotient map
\[\pi:\mathbf{A}^{2d}_{\C}-H\to F_d\]
is algebraic and thus induces a morphism of mixed Hodge structures $\pi^*:H^*(F_d,\Q)\to H^*(\mathbf{A}^{2d}_{\C}-H,\Q)$. Moreover, it is classical that $\pi^*$ is injective (see e.g. \cite[Theorem III.2.4]{Bredon}). Since the mixed Hodge structure of $H^k(\mathbf{A}^{2d}_{\C}-H,\Q)$ is pure of weight $2k$ (by Proposition \ref{prop: purity for subspace arrangements} or by \cite{kimweights}), the desired result follows.
\end{proof}

In \cite[Proposition 3.1.]{cazanavealgebraic}, Cazanave shows that the variety $\bigsqcup_dF_d$ has the structure of a graded monoid in $\on{Sm}_\C$. The structure of a graded monoid can be encoded by a colored operad. Thus the following result follows from Theorem \ref{theo: main covariant}.

\begin{theo}\label{theo: monoidformal}
The graded monoid in chain complexes $\bigoplus_dS_*(F_d,\Q)$ is formal.
\end{theo}

\subsection{The little disks operad}\label{subsect: little disks}

In \cite{Petersen}, Petersen shows that the operad of little disks $\oper{D}$ is formal. The method of proof is to use the action of a certain group $\on{GT}(\Q)$ on $S_*(\oper{PAB}_\Q,\Q)$ which follows from work of Drinfeld. Here the operad $\oper{PAB}_\Q$ is rationally equivalent to $\oper{D}$ and $\on{GT}(\Q)$ is the group of $\Q$-points of the pro-algebraic Grothendieck-Teichmüller group. We can reinterpret this proof using the language of mixed Hodge structures. Indeed, the group $\on{GT}$ receives a map from the group $\on{Gal}(\cat{MT}(\mathbb{Z}))$, the Galois group of the Tannakian category of mixed Tate motives over $\mathbb{Z}$ (see \cite[25.9.2.2]{andrebook}). Moreover there is a map $\on{Gal}(\cat{MHTS}_\Q)\to\on{Gal}(\cat{MT}(\mathbb{Z}))$ from the Tannakian Galois group of the abelian category of mixed Hodge Tate structures (the full subcategory of $\cat{MHS}_\Q$ generated under extensions by the Tate twists $\Q(n)$ for all $n$) which is Tannaka dual to the tensor functor
\[\cat{MT}(\Z)\lra \cat{MHTS}_\Q\]
sending a mixed Tate motive to its Hodge realization. This map of Galois group allows us to view $S_*(\oper{PAB}_\Q,\Q)$ as an operad in $\ch(\cat{MHS}_\Q)$ which moreover has a $2$-pure weight filtration (as follows from the computation in \cite{Petersen}). Therefore by Corollary \ref{coro-purity formality Q}, the operad $S_*(\oper{PAB}_\Q,\Q)$ is formal and hence also $S_*(\oper{D},\Q)$.

\subsection{The gravity operad}

In \cite{dupontgravity}, Dupont and the second author prove the formality of the gravity operad of Getzler. It is an operad structure on the collection of graded vector spaces $\{H_{*-1}(\mathcal{M}_{0,n+1}), n\in \mathbb{N}\}$. It can be defined as the homotopy fixed points of the circle action on $S_*(\oper{D},\Q)$. The method of proof in \cite{dupontgravity} can also be interpreted in terms of mixed Hodge structures. Indeed, a model $\oper{G}rav^{W'}$ of gravity is constructed in 2.7 of loc. cit. This model comes with an action of $\on{GT}(\Q)$ and a $\on{GT}(\Q)$-equivariant map $\iota:\oper{G}rav^{W'}\lra S_*(\oper{PAB}_\Q,\Q)$ which is injective on homology. As in the previous subsection, this action of $\on{GT}(\Q)$ allows us to interpret $\oper{G}rav^{W'} $ as an operad in $\ch(\MHS_\Q)$. Moreover, the injectivity of $\iota$ implies that $\oper{G}rav^{W'}$ also has a $2$-pure weight filtration. Therefore by Corollary \ref{coro-purity formality Q}, we deduce the formality of $\oper{G}rav^{W'}$. In fact, we obtain the stronger result that the map
\[\iota:\oper{G}rav^{W'}\lra S_*(\oper{PAB}_\Q,\Q)\]
is formal as a map of operads (i.e. it is connected to the induced map in homology by a zig-zag of maps of operads).

\subsection{$E^1$-formality}
The above results deal with objects whose weight filtration is pure. 
In general, for mixed weights, the singular chains functor is not formal, but it is $E^1$-formal as we now explain.

The $r$-stage of the spectral sequence associated to a filtered complex is an $r$-bigraded complex with differential of bidegree $(-r,r-1)$.
By taking its total degree and considering the column filtration we obtain a filtered complex.
Denote by
$$E^r:\ich{(\Ff\Q)}\lra \ich{(\Ff\Q)}$$
the resulting strong symmetric monoidal $\infty$-functor.
Denote by 
$$\tilde \Pi^{W}_\Q:\iMHC_\Q\lra \ich(\Ff\Q)$$ the forgetful functor 
defined by sending a mixed Hodge complex to its rational component together with the weight filtration.
Note that, since the weight spectral sequence of a mixed Hodge complex degenerates at the second stage, 
the homology of $E^1\circ \tilde \Pi^{W}_\Q$ gives the weight filtration on the homology of 
mixed Hodge complexes. We have:

\begin{theo}\label{E1formality}
Denote by $S^{fil}_*:\inn (\Var_\C)\lra \ich(\Q)$ the composite functor
$$\inn(\Var_\C)\xrightarrow{\D_*}\iMHC_\Q\xrightarrow{\tilde \Pi^{W}_\Q}\ich(\Ff\Q).$$
There is an equivalence of strong symmetric monoidal $\infty$-functors
$E^1\circ S_*^{fil}\simeq S^{fil}_*.$
\end{theo}

\begin{proof}
It suffices to prove an equivalence $\tilde \Pi^{W}_\Q\simeq E^1\circ \tilde \Pi^{W}_\Q$.
We have a commutative diagram of strong symmetric monoidal $\infty$-functors.
$$
\xymatrix{
\ich(\MHS_\Q)\ar[d]_{\Pi^{W}_\Q}\ar[r]^{\Tt}&\iMHC_\Q\ar[d]^{\tilde \Pi^{W}_\Q}\\
\ich(\Ff\Q)\ar[r]^{T}\ar[d]_{E^0}&\ich(\Ff\Q)\ar[d]^{E^1}\\
\ich(\Ff\Q)\ar[r]^{T}&\ich(gr\Q)
}
$$
The commutativity of the top square follows from the definition of $\Tt$. We prove that the bottom square commutes.
Recall that $T(K,W)$ is the filtered complex $(K,TW)$ 
defined by $TW^pK_n:=W^{p+n}K_n$. It satisfies $d(TW^pK_p)\subset TW^{p+1}K_{n-1}$. In particular, the induced differential on 
$Gr_{TW}K$ is trivial. Therefore we have:
$$E^1_{-p,q}(K,TW)\cong H_{q-p}(Gr^{p}_{TW}K)\cong Gr^{p}_{TW}K_{q-p}=Gr^q_WK_{q-p}=E^0_{-q,2q-p}(K,W).$$
This proves that the above diagram commutes.

Since $\Tt$ is an equivalence of $\infty$-categories, it is enough to prove that $E^1\circ \tilde \Pi^{W}_\Q\circ \Tt$ is equivalent 
to $\tilde  \Pi^{W}_\Q\circ \Tt$. By the commutation of the above diagram it suffices to prove that there is an equivalence $E^0\circ \Pi^{W}_\Q\cong \Pi^{W}_\Q$. This follows from Lemma \ref{Deligne_splitting_over_Q},
since $E^0=U^{fil}\circ gr$.
\end{proof}

\section{Rational homotopy of varieties and formality}\label{Section_formalschemes}

For $X$ a space, we denote by $\Aa_{PL}^*(X)$, Sullivan's algebra of piecewise linear differential forms.
This is a commutative dg-algebra over $\Q$ that captures the rational homotopy type of $X$.
A contravariant version of Theorem $\ref{theo: main covariant}$ gives:

\begin{theo}\label{theo: main contravariant}
Let $\alpha$ be a non-zero rational number. The functor
\[\Aa_{PL}^*:\Var_\C\op\lra \ch(\Q)\]
is formal as a lax symmetric monoidal functor when restricted to varieties whose weight filtration in cohomology is $\alpha$-pure.
\end{theo}

\begin{proof}
The proof is the same as the proof of Theorem $\ref{theo: main covariant}$ using $\D^*$ instead of $\D_*$ and using the fact that $\D^*(-)_\Q$ is quasi-isomorphic to $\Aa_{PL}^*$ as a lax symmetric monoidal functor (see \cite[Th\'{e}or\`{e}me 5.5]{Na}).
\end{proof}

Recall that a topological space $X$ is said to be \textit{formal} if there is a string of quasi-isomorphisms of 
commutative dg-algebras from $\Aa_{PL}^*(X)$ to $H^*(X,\Q)$, where $H^*(X,\Q)$ is considered as a commutative dg-algebra with trivial differential.
Likewise, a continuous map of topological spaces $f:X\lra Y$ is \textit{formal} if there is a string of homotopy commutative diagrams of morphisms
$$
\xymatrix{
\Aa_{PL}^*(Y)\ar[d]^{f^*}&\ar[l]\ast\ar[d]&\ar[l]\cdots\ar[r]&\ast\ar[d]\ar[r]&H^*(Y,\Q)\ar[d]^{H^*(f)}\\
\Aa_{PL}^*(X)&\ar[l]\ast&\cdots\ar[l]\ar[r]&\ast\ar[r]&H^*(X,\Q)
}
$$
where the horizontal arrows are quasi-isomorphisms. 
Note that if $f:X\to Y$ is a map of topological spaces
and $X$ and $Y$ are both formal spaces, then it is not always true that $f$ is a formal map. 
Also, in general, the composition of formal morphisms is not formal.

Theorem $\ref{theo: main contravariant}$
gives functorial formality for varieties with pure weight filtration in cohomology,
generalizing both ``purity implies formality'' statements appearing in \cite{Dupont} for smooth varieties
and in \cite{ChCi1} for singular projective varieties. We also get a result of partial formality as 
done in these references, via Proposition \ref{decomposition_qformal}. 
Our generalization is threefold, as explained in the following three subsections.

\subsection{Rational purity}
To our knowledge, in the existing references where $\alpha$-purity of the weight filtration is discussed,
only the cases $\alpha=1$ and $\alpha=2$ are considered, whereas we obtain formality for varieties with $\alpha$-pure cohomology,
for $\alpha$ an arbitrary non-zero rational number. We now show that certain complement of subspaces arrangements give examples of such varieties.

\begin{defi}
Let $V$ be a finite dimensional $\C$-vector space. We say that a finite set $\{H_i\}_{i\in I}$ of subspaces of $V$ is a \textit{good arrangement of codimension $d$ subspaces} if 
\begin{enumerate}
\item [(i)]For each $i \in I$, the subspace $H_i$ is of codimension $d$.
\item [(ii)]For each $i\in I$, the set of subspaces $\{H_i\cap H_j\}_{j\neq i}$ of $H_i$ form a good arrangement of codimension $d$ subspaces.
\end{enumerate}
\end{defi}

\begin{rem}
In particular the empty set of subspaces is a good arrangement of codimension $d$ subspaces. By induction on the size of $I$, we see that this condition is well-defined.
\end{rem}

\begin{example}
Recall that a set of subspaces of codimension $d$ of an $n$-dimensional space is said to be in general position if the intersection of $k$ of those subspaces is of codimension $\on{min}(n,dk)$. One easily checks that a set of codimensions $d$ subspaces in general position is a good arrangement. However, the converse does not hold as shown in the following example.
\end{example}

\begin{example}
Take $V=(\C^d)^m$ and define, for $(i,j)$ an unordered pair of distinct elements in $\{1,\ldots,m\}$, the subspace
\[W_{(i,j)}=\{(x_1,\ldots,x_n)\in (\C^d)^m, x_i=x_j\}.\]
This collection of codimension $d$ subspaces of $V$ is a good arrangement. However, these subspaces are not in general position if $m$ is at least $3$. Indeed, the codimension of $W_{(1,2)}\cap W_{(1,3)}\cap W_{(2,3)}$ is $2d$. The complement $V-\bigcup_{(i,j)}W_{(i,j)}$ is exactly $F_m(\C^d)$, the space of configurations of $m$ points in $\C^d$.
\end{example}

\begin{prop}\label{prop: purity for subspace arrangements}
Let $H=\{H_1,\ldots,H_k\}$ be a good arrangement of codimension $d$ subspaces of $\C^n$. Then the mixed Hodge structure on  $H^*(\C^n-\cup_i H_i,\Q)$  is pure of weight $2d/(2d-1)$.
\end{prop}

\begin{proof}
We proceed by induction on $k$. This is obvious for $k=0$. Now, we consider the variety $X=\C^n-\cup_i^{k-1} H_i$, It contains an open subvariety $U=\C^n-\cup_i^{k} H_i$ and its closed complement $Z=H_k-\cup_i^{k-1} H_i\cap H_k$ which has codimension $d$. Therefore the purity long exact sequence on cohomology groups has the form
\[\ldots \lra H^{r-2d}(Z)(-d)\lra H^r(X)\lra H^r(U)\lra H^{r+1-2d}(Z)(-d)\lra \ldots\]
By the induction hypothesis, the Hodge structure on $H^{r+1-2d}(Z)(-d)$ and on $H^r(X)$ are pure of weight $2dr/(2d-1)$ and hence it is also the case for $H^r(U)$ as desired.
\end{proof}

\begin{rem}
This proposition is well-known for $d=1$ and is proved for instance in \cite{kimweights}. Note that the space $F_m(\C^d)$ of configurations of $m$ points in $\C^d$ fits in the above proposition, so we 
recover formality of these spaces by purely Hodge theory arguments.
\end{rem}

\subsection{Functoriality}Every morphism of smooth complex projective varieties is formal. However, if $f:X\to Y$ is an algebraic morphism of complex varieties (possibly singular and/or non-projective), and both $X$ and $Y$ are formal, the morphism $f$ need not be formal.

\begin{example}
Consider the algebraic  Hopf fibration
$f:\C^2\setminus\{0\}\lra \mathbf{P}^1_{\C}$ defined by $(x_0,x_1)\mapsto [x_0:x_1]$.
Both spaces $\C^2\setminus\{0\}\simeq S^3$ and $\mathbf{P}^1_{\C}\simeq S^2$ are formal. The morphism induced by $f$ in cohomology is trivial in all positive degrees. Therefore, if $f$ were formal, this would mean that $f$ is nullhomotopic. However, it is well-known that $f$ generates the one dimensional vector space $\pi_3(S^2)\otimes\Q$. Note in fact, that $\mathbf{P}^1_{\C}$ has $1$-pure weight filtration while $\C^2\setminus\{0\}$ has $4/3$-pure weight filtration.
\end{example}

Theorem $\ref{theo: main contravariant}$ tells us that if $f:X\lra Y$ is a morphism of algebraic varieties and both $X$ and $Y$ 
have $\alpha$-pure cohomology, with $\alpha$ a non-zero rational number (the same $\alpha$ for $X$ and $Y$), then $f$ is a formal morphism.
This generalizes the formality of holomorphic morphisms between compact K\"{a}hler manifolds of \cite{DGMS}
and enhances the results of \cite{Dupont} and \cite{ChCi1} by providing them with functoriality. In fact, we have:

\begin{prop}
Let $f:X\lra Y$ be a morphism between connected complex varieties. Assume that the weight filtration on the cohomology of $X$ (resp. $Y$) is $\alpha$-pure (resp. $\beta$-pure). Then:
\begin{enumerate}
 \item If $\alpha=\beta$, then $f$ is formal.
 \item If $\alpha\neq\beta$, then $f$ is formal only if it is rationally nullhomotopic.
\end{enumerate}
\end{prop}

\begin{proof}
Let us first give the precise definition that we will use of a rationally nullhomotopic map. We say that a map $g:U\lra V$ between topological spaces is rationally nullhomotopic if the induced map
\[\Aa_{PL}(g):\Aa_{PL}^*(V)\lra \Aa_{PL}^*(U)\]
is equal in the homotopy category of cdga's to a map that factors through the initial cdga $\Q$.

When $\alpha=\beta$, Theorem $\ref{theo: main contravariant}$ ensures that $f$ is formal. 

If $\alpha\neq \beta$, then we claim that $H^*(f)$ is zero in positive degree. Indeed, since $H^*(f)$ is strictly compatible with the weight filtration, it suffices to show that the morphism
$$Gr_{p}^WH^n(Y,\Q)\lra Gr_{p}^WH^n(X,\Q)$$ is trivial for all $p\in\Z$ and all $n>0$ which follows from the purity conditions. Therefore, if $f$ is formal, the map $\Aa_{PL}^*(f)$ coincides with $H^*(f)$ in the homotopy category of cdga's and the latter map factors through $\Q$. 
\end{proof}

\subsection{Non-projective singular varieties}

The following result of formality of non-projective singular complex varieties with pure Hodge structure seems to be a new result. 
\begin{example}
Let $X$ be an irreducible singular projective variety
of dimension $n>0$ with $1$-pure weight filtration in cohomology.
Let $p\in X$ be a smooth point of $X$. Then, we claim that the complement $X-p$ has $1$-pure weight filtration in cohomology.
Indeed, we can consider the long exact sequence of cohomology groups for the pair $(X,X-p)$. 
\[\cdots\to H^{i-1}(X-p)\to H^i(X,X-p)\to H^i(X)\to H^i(X-p)\to H^{i+1}(X,X-p)\to\cdots \]
Since $p$ is a smooth point, there exists a neighborhood $U$ of $p$ that is homeomorphic to $\mathbb{R}^{2n}$, therefore excision gives us an isomorphism
\[H^k(X,X-p)\cong H^k(U,U-p)\]
Since $H^k(U,U-p)$ is non-zero only when $k=2n$, we deduce that the map $H^k(X)\to H^k(X-p)$ is an isomorphism for all $k<2n-1$.
Moreover, since $X$ is irreducible, we have $H^{2n}(X)=\Q$ and this vector space has a generator, the fundamental class, which is in the image of $H^{2n}(X,X-q)\to H^{2n}(X)$ for any smooth point $q$. Together with the above long exact sequence, this implies that $H^{2n-1}(X-p)\cong H^{2n-1}(X)$ and $H^{2n}(X-p)=0$. To summarize, we have proved that the inclusion $X-p\to X$ induces an isomorphism on all cohomology groups except on the top one where $H^{2n}(X)=\Q$ while $H^{2n}(X-p)=0$. This proves that the weight filtration of $X-p$ is $1$-pure. As a consequence, the space $X-p$ is formal and the inclusion $X-p\hookrightarrow X$ is formal.
\end{example}

\subsection{$E_1$-formality}
We also have a contravariant version of Theorem $\ref{E1formality}$.

\begin{theo}\label{E1formality_contravariant}
Denote by $\Aa_{fil}^*:\inn(\Var_\C)\op\lra \ich(\Ff\Q)$ the composite functor
$$\inn(\Var_\C)\op\xrightarrow{\D^*}\iMHC_\Q\xrightarrow{\tilde \Pi^{W}_\Q}\ich(\Ff\Q).$$
Then 
\begin{enumerate}
\item The lax symmetric monoidal $\infty$-functors $\Aa_{fil}^*$ and $E_1\circ \Aa_{fil}^*$ are weakly equivalent.

\item  Let $U:\ch(\Ff\Q)\lra \ch(\Q)$ denote the forgetful functor. 
The lax symmetric monoidal $\infty$-functor
$U\circ E_1\circ \Aa_{fil}^*:\inn(\Var_\C)\op\to \ich(\Q)$ is weakly equivalent to Sullivan's functor $\Aa_{PL}^*$ of piecewise linear forms. 

\item The lax symmetric monoidal functor
$U\circ E_1\circ \Aa_{fil}^*:\Sm_\C\op\to \ch(\Q)$ is weakly equivalent to Sullivan's functor $\Aa_{PL}^*$ of piecewise linear forms. 
\end{enumerate}
\end{theo}

\begin{proof}
The first part is proven as Theorem \ref{E1formality} replacing $\D^*$ by $\D_*$. The second part follows from the first part and the fact that $\Aa_{PL}^*(-)$ is naturally weakly equivalent to $\D^*(-)_{\Q}\simeq U\circ \Aa^*_{fil}$ (Proposition \ref{prop:equivalence D S}). The third part follows from the second part and Theorem \ref{theo: Hinich rigidification}, using the fact that both functors are ordinary lax symmetric monoidal functors when restricted to smooth varieties.
\end{proof}

\begin{rem}In \cite{Mo} it is proven that the complex homotopy type of every smooth complex variety is $E_1$-formal. This is extended to possibly singular varieties and their morphisms in \cite{CG1}. Then, a descent argument is used to prove that for nilpotent spaces (with finite type minimal models), this result descends to the rational homotopy type. Theorem $\ref{E1formality_contravariant}$ enhances the contents of \cite{CG1} in two ways: first, since descent is done at the level of functors, we obtain $E_1$-formality over $\Q$ for any complex variety, without nilpotency conditions (the only property needed is finite type cohomology). Second, the functorial nature of our statement makes $E_1$-formality at the rational level, compatible with composition of morphisms.
\end{rem}

\subsection{Formality of Hopf cooperads}

Our main theorem takes two dual forms, one covariant and one contravariant. The covariant theorem yields formality for algebraic structures (like monoids, operads, etc.), the contravariant theorem yields formality for coalgebraic structure (like the comonoid structure coming from the diagonal $X\to X\times X$ for any variety $X$). One might wonder if there is a way to do both at the same time. For example, if $M$ is a topological monoid, then $H^*(M,\Q)$ is a Hopf algebra where the multiplication comes from the diagonal of $M$ and the comultiplication comes from the multiplication of $M$. One may ask whether $S^*(M,\Q)$ is formal as a Hopf algebra. This question is not well-posed because $S^*(M,\Q)$ is not a Hopf algebra on the nose. The problem is that there does not seem to exist a model for singular chains or cochains that is strong symmetric monoidal: the standard singular chain functor $S_*(-,\Q)$ is lax symmetric monoidal and Sullivan's functor $\Aa_{PL}^*$ is oplax symmetric monoidal functor from $\cat{Top}$ to $\ch(\Q)\op$. 

Nevertheless, the functor $\Aa_{PL}^*$ is strong symmetric monoidal ``up to homotopy''. It follows that, if $M$ is a topological monoid, $\Aa_{PL}^*(M)$ has the structure of a cdga with a comultiplication up to homotopy and it makes sense to ask if it has formality as such an object. In order to formulate this more precisely, we introduce the notion of an algebraic theory. The following is inspired by Section 3 of \cite{LaVo}.

\begin{defi}
An algebraic theory is a small category $T$ with finite products. For $\Cc$ a category with finite products, a $T$-algebra in $\Cc$ is a finite product preserving functor $T\lra\Cc$.
\end{defi}

There exist algebraic theories for which the $T$-algebras are monoids, groups, rings, operads, cyclic operads, modular operads etc. 

\begin{rem}
Definitions of algebraic theories in the literature are usually more restrictive. This definition will be sufficient for our purposes.
\end{rem}

\begin{defi}
Let $T$ be an algebraic theory. Let $\kk$ be a field. Then a dg Hopf $T$-coalgebra over $\kk$ is a finite coproduct preserving functor from $T\op$ to the category of cdga's over $\kk$.
\end{defi}

\begin{rem}
Recall that the coproduct in the category of cdga's is the tensor product. It follows that a dg Hopf $T$-coalgebra for $T$ the algebraic theory of monoids is a dg Hopf algebra whose multiplication is commutative. A dg Hopf $T$-coalgebra for $T$ the theory of operads is what is usually called a dg Hopf cooperad in the literature.
\end{rem}

\begin{defi}
Let $T$ be an algebraic theory and $\Cc$ a category with products and with a notion of weak equivalences. A weak $T$-algebra in $\Cc$ is a functor $F:T\lra \Cc$ such that for each pair $(s,t)$ of objects of $T$, the canonical map
\[F(t\times s)\lra F(t)\times F(s)\]
is a weak equivalence. A weak $T$-algebra in the opposite category of $\cat{CDGA}_\kk$ is called a weak dg Hopf $T$-coalgebra.
\end{defi}

Observe that if $X:T\lra \cat{Top}$ is a $T$-algebra in topological spaces (or even a weak $T$-algebra), then $\Aa^*_{PL}(X)$ is a weak dg Hopf $T$-coalgebra. Our main theorem for Hopf $T$-coalgebras is the following.

\begin{theo}
Let $\alpha$ be a rational number different from zero. Let $X:T\lra \Var_\C$ be a $T$-algebra such that for all $t\in T$, the weight filtration on the cohomology of $X(t)$ is $\alpha$-pure. Then $\Aa_{PL}^*(X)$ is formal as a weak dg Hopf $T$-coalgebra.
\end{theo}

\begin{proof}
Being a weak $T$-coalgebra is a property of a functor $T\op\to \cat{CDGA}_\kk$ that is invariant under quasi-isomorphism. Thus the result follows immediately from  Theorem \ref{theo: main contravariant}.
\end{proof}

It should be noted that knowing that $\Aa_{PL}^*(X)$ is formal as a dg Hopf $T$-coalgebra implies that the data of $H^*(X,\Q)$ is enough to reconstruct $X$ as a $T$-algebra up to rational equivalence. Indeed, recall the Sullivan spatial realization functor
\[\langle-\rangle : \cat{CDGA}_\kk \lra \cat{Top}\] 
Applying this functor to a weak dg Hopf $T$-coalgebra yields a weak $T$-algebra in rational spaces. Specializing to $\Aa^*_{PL}(X)$ where $X$ is a $T$-algebra in spaces, we get a rational model for $X$ in the sense that the map
\[X\lra \langle \Aa^*_{PL}(X)\rangle\]
is a rational weak equivalence of weak $T$-algebras whose target is objectwise rational. It should also be noted that for reasonable algebraic theories $T$ (including in particular the theory for monoids, commutative monoids, operads, cyclic operads), the homotopy theory of $T$-algebras in spaces is equivalent to that of weak $T$-algebras by the main theorem of \cite{bergnerrigidification}. In particular our weak $T$-algebra $\langle \Aa^*_{PL}(X)\rangle$ can be strictified to a strict $T$-algebra that models the rationalization of $X$. If $\Aa^*_{PL}(X)$ is formal, one also get a rational model for $X$ by applying the spatial realization to the strict Hopf $T$-coalgebra $H^*(X,\Q)$. Thus the rational homotopy type of $X$ as a $T$-algebra is a formal consequence of $H^*(X,\Q)$ as a Hopf $T$-coalgebra.

\begin{example}
Applying this theorem to the non-commutative little disks operad and framed little disks operad of subsection \ref{subsection: ncld}, we deduce that $\Aa^*_{PL}(\Aa s_{S^1})$  and $\Aa^*_{PL}(\Aa s_{S^1}\rtimes S^1)$ are formal as a weak Hopf non-symmetric cooperads. Similarly applying this to the monoid of self maps of the projective line of subsection \ref{subsection : self maps}, we deduce that $\Aa^*_{PL}(\bigsqcup_d F_d)$ is formal as a weak Hopf graded comonoid.
\end{example}

\bibliographystyle{alpha}

\bibliography{biblio}

\end{document}